\documentclass[12pt]{article}
\usepackage[margin=1.2in]{geometry}

\usepackage{graphicx}
\usepackage{amssymb}
\usepackage{amsmath}
\usepackage{graphicx}
\usepackage{wrapfig}
\usepackage{authblk}
\usepackage[all]{xy}
\usepackage{bbm} 
\usepackage{float}
\usepackage{tikz}
\usetikzlibrary{arrows,calc,shapes,decorations.pathreplacing}
\usepackage{caption}
\usepackage{pgfplots}
\usetikzlibrary{patterns}
\usepackage{subcaption} 

\newenvironment{customlegend}[1][]{%
    \begingroup
    \csname pgfplots@init@cleared@structures\endcsname
    \pgfplotsset{#1}%
}{%
    \csname pgfplots@createlegend\endcsname
    \endgroup
}%

\usepackage{algorithm}
\usepackage{algorithmicx}
\usepackage{algpseudocode}

\def\addlegendimage{\csname pgfplots@addlegendimage\endcsname}

\newtheorem{theorem}{Theorem}[section] 
\newtheorem{lemma}[theorem]{Lemma}
\newtheorem{proposition}[theorem]{Proposition}
\newtheorem{corollary}[theorem]{Corollary}

\newtheorem{example}{Example}[section]

\newenvironment{proof}[1][Proof]{\begin{trivlist}
\item[\hskip \labelsep {\bfseries #1}]}{\end{trivlist}}

\newenvironment{remark}[1][Remark]{\begin{trivlist}
\item[\hskip \labelsep {\bfseries #1}]}{\end{trivlist}}

\newcommand{\qed}{\nobreak \ifvmode \relax \else
      \ifdim\lastskip<1.5em \hskip-\lastskip
      \hskip1.5em plus0em minus0.5em \fi \nobreak
      \vrule height0.75em width0.5em depth0.25em\fi}
\newcommand{\argmin}{\operatornamewithlimits{argmin}}

\renewcommand{\arraystretch}{2}

\numberwithin{equation}{section}
\begin{document}

\title{Scheduling for a Processor Sharing System\\ 
with Linear Slowdown\footnote{To appear in Mathematical Methods of Operations Research}}
\author{Liron Ravner\footnote{Department of Statistics and the Federmann Center for the Study of Rationality, The Hebrew University of Jerusalem, Israel} ~and Yoni Nazarathy\footnote{School of Mathematics and Physics, The University of Queensland, Brisbane, Australia}}

\date{\today}
\maketitle

\begin{abstract}
We consider the problem of scheduling arrivals to a congestion system with a finite number of users having identical deterministic demand sizes. The congestion is of the processor sharing type in the sense that all users in the system at any given time are served simultaneously. However, in contrast to classical processor sharing congestion models, the processing slowdown is proportional to the number of users in the system at any time. That is, the rate of service experienced by all users is linearly decreasing with the number of users. For each user there is an ideal departure time (due date). A centralized scheduling goal is then to select arrival times so as to minimize the total penalty due to deviations from ideal times weighted with sojourn times. Each deviation is assumed quadratic, or more generally convex. But due to the dynamics of the system, the scheduling objective function is non-convex. Specifically, the system objective function is a non-smooth piecewise convex function. Nevertheless, we are able to leverage the structure of the problem to derive an algorithm that finds the global optimum in a (large but) finite number of steps, each involving the solution of a constrained convex program. Further, we put forward several heuristics. The first is the traversal of neighbouring constrained convex programming problems, that is guaranteed to reach a local minimum of the centralized problem. This is a form of a ``local search'', where we use the problem structure in a novel manner. The second is a one-coordinate ``global search'', used in coordinate pivot iteration. We then merge these two heuristics into a unified ``local-global'' heuristic, and numerically illustrate the effectiveness of this heuristic.
\end{abstract}

\section{Introduction}\label{sec:Intro}

Users of shared resources are frequently faced with the decision of when to use the resource with a view of trying to avoid rush hour effects. Broad examples include, workers taking their lunch break and attending a cafeteria; people entering and vacating sporting events; and commuters using transportation networks. In many such situations the so called rush-hour game is played by all users acting individually. On the one hand, each user typically has an ideal arrival/departure time, while on the other hand, users often wish to avoid rush hour so as to minimise congestion costs. These general types of scenarios have received much attention through the transportation community, \cite{ADL1993}, the queueing community (see \cite{GH1983} or p84 of \cite{book_H2016} for a review) and more specifically within the setting we consider in this paper \cite{RHV2016}.

While understanding social strategic (game) behaviour is important, a complementary analysis is with regards to the social optimum (centralised scheduling decisions). These types of situations occur often in manufacturing, appointment scheduling, education and service. Most of the research on scheduling methodology does not consider processor sharing but rather focuses on the situation where resources are dedicated, see \cite{book_P2008}. In this paper, we put forward a novel scheduling model, that offers a simple abstraction of a common scenario: Jobs may be scheduled simultaneously, yet slow each other down when sharing the resource. In this respect our model is related to the study of scheduling problems with batch processing, see \cite{PM2000}. However, from a mathematical perspective, our model, results and methods do not involve the classical discrete approaches but rather rely on piecewise affine dynamics with breakpoints. This type of behaviour resembles Separated Continuous Linear Programs, as in \cite{weiss2008simplex}, and is often used to solve optimization problems associated with fluid multi-class queueing networks (cf. \cite{avram1995fluid}, \cite{nazarathy2009near}). 

A standard way of modelling resource sharing phenomena, is the so-called processor sharing queue, see for example \cite{harchol2013performance}. In such a model, given that at time $t$ there are $q(t)$ users in the system, the total fixed service capacity, $\beta>0$, is allocated, such that each user receives an instantaneous {\em service rate},
\begin{equation}
\label{eq:56}
v\big(q(t) \big) = \frac{\beta}{q(t)}.
\end{equation}
Such a model then captures the relationship of the arrival time of a user, $a$, the departure time of a user, $d$ and the service demand, $\ell$ through
\[
\ell = \int_{a}^{d} v\big(q(t)\big) dt.
\]
The {\em aggregate throughput} with $q$ users in the system is the product $q \, v(q)$. For the processor sharing model \eqref{eq:56}, this is obviously the constant $\beta$. However, in practice, the aggregate throughput is not necessarily constant with respect to $q(t)$. In many situations, most notably in traffic and transportation scenarios, users inter-play in a complicated manner. In particular, in the classic Greenshield fluid model, (see for example \cite{H1974} or \cite{MAHMASSANI1984}) the aggregate throughput is not monotone in the number of users and even exhibits a traffic jam effect. The simplest model, describing such a phenomenon~is
\begin{equation}
\label{eq:ourSpeed}
v\big(q(t)\big) = \beta - \alpha \big(q(t)-1\big),
\end{equation}
which is a discrete variation of Greenshield's model\footnote{Note that in queueing theory, situations where $v(\cdot)$ is not as in \eqref{eq:56} but is rather some other function are sometimes referred to as generalized processor sharing. See for example~\cite{cohen1979multiple}. Generalized processor sharing has also taken other meanings over the years, so sometimes there is confusion about the term.}. With a single user in the system, \eqref{eq:ourSpeed} yields the {\em free flow rate} $\beta$ which coincides with \eqref{eq:56}. Then for each additional user, there is a linear slowdown of $\alpha>0$ units in the rate. See Figure~\ref{fig:service_dynamics} for a simple illustration. Note that in road networks, much research has focused on the so-called fundamental diagram for networks, such as in \cite{Daganzo2007}. Indeed Figure~\ref{fig:service_dynamics}-b resembles a fundamental diagram. 
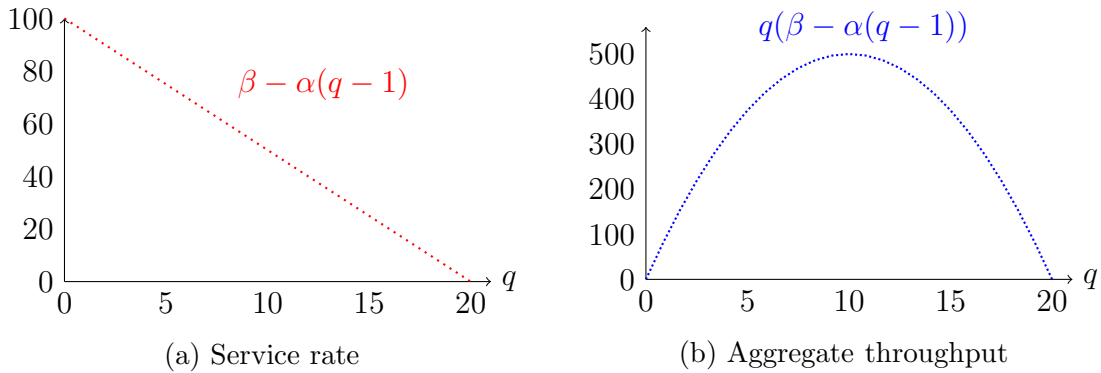
\begin{figure}
\begin{subfigure}{0.49\linewidth}
\centering
\begin{tikzpicture}[xscale=0.27,yscale=0.035]
  \def\xmin{0}
  \def\xmax{21}
  \def\ymin{0}
  \def\ymax{100}
  
    \draw[->] (\xmin,\ymin) -- (\xmax,\ymin) node[right] {$q$} ;
    \draw[->] (\xmin,\ymin) -- (\xmin,\ymax);
    \foreach \x in {0,5,10,15,20}
    \node at (\x,\ymin) [below] {\x};
    \foreach \y in {0,20,40,60,80,100}
    \node at (\xmin,\y) [left] {\y};
  
  \draw[red, dotted, thick,domain=0:20]  plot (\x, {100-5*\x});
  \draw[] (8,75) node[right,red] {$\beta-\alpha(q-1)$};  
\end{tikzpicture}
\caption{Service rate}
\end{subfigure}
\begin{subfigure}{0.49\linewidth}
\centering
\begin{tikzpicture}[xscale=0.27,yscale=0.006]
  \def\xmin{0}
  \def\xmax{21}
  \def\ymin{0}
  \def\ymax{560}
  
    \draw[->] (\xmin,\ymin) -- (\xmax,\ymin) node[right] {$q$} ;
    \draw[->] (\xmin,\ymin) -- (\xmin,\ymax);
        \foreach \x in {0,5,10,15,20}
    \node at (\x,\ymin) [below] {\x};
    \foreach \y in {0,100,200,300,400,500}
    \node at (\xmin,\y) [left] {\y};
  
  \draw[blue, densely dotted, thick,domain=0:20]  plot (\x, {\x*(100-5*\x)});
  \draw[] (5,560) node[right,blue] {$q(\beta-\alpha(q-1))$};
  
\end{tikzpicture}
\caption{Aggregate throughput}
\end{subfigure}
\caption{The service rate and aggregate throughput as a function of the number of users in the system. Parameter values: $\beta=100$ and $\alpha=5$}
\label{fig:service_dynamics}
\end{figure}

Our scheduling problem is to centrally choose arrival times ${\mathbf a}=(a_1,\ldots,a_N)'$ in an effective manner, where $N$ is the number of users. In this paper we assume that all users share the same service demand, $\ell$. In our objective, user $i$ incurs a cost of
\[
(d_i-d_i^*)^2 \,+ \gamma \,  (d_i - a_i),
\]
where $d_i$ is his departure time and $d_i^*$ is the ideal departure time (due date) and $\gamma $ captures tradeoff between meeting the due date and sojourn time costs. The total costs incurred by all users is then the sum of individual user costs.

If there was no congestion (say due to $d_i^*$ being well separated), an ideal choice is $a_i = d_i^* - \ell/\beta$. But in general, users interact, so the scheduling decision needs to take this interaction into account. If, for example, $\gamma=0$ and $d_i^*=d^*$ for all $i$, then the problem is trivially solved with zero cost by setting
\[
a_i = d^*- \frac{\ell}{\beta- \alpha(N-1)}.
\]
Here since sojourn time does not play a role, sending all users simultaneously will imply they arrive simultaneously after being served together at the slowest possible rate.  Continuing with the case of $\gamma=0$, if now users do not have the same $d_i^*$, then attaining zero costs is still possible. In fact, we show in the sequel, that in this specific case ($\gamma=0$) the optimal schedule can be computed efficiently (in polynomial time).  

At the other extreme consider the case where minimising sojourn times is prioritised over minimisation of due dates (e.g. if fuel costs are extremely high). This corresponds to $\gamma \approx \infty$. While for any finite $\gamma$, it is possible that an optimal schedule allows overlap of users, an approximation for the case of large $\gamma$ is obtained by enforcing a schedule with no overlap ($q(t) \le 1 ~ \forall t$). This is because overlaps have a very large sojourn time cost relative to the possible reduction in quadratic deviation from desired departure times. Now with such a constraint, the problem resembles a single machine scheduling problem with due date penalties. This problem has been heavily studied (see for example \cite{baker1990sequencing} or \cite{sen1984state}).  In our case, in which users have identical demand, finding the optimal schedule is a convex quadratic program and can thus be solved in polynomial time. We spell out the details in the sequel.

Setting aside the extreme cases of $\gamma=0$ or $\gamma \approx \infty$, the problem is more complicated. While we do not have an NP-hardness proof, we conjecture that finding the optimal ${\mathbf a}$ is a computationally challenging problem. In the current paper we handle this problem in several ways. First we show that departure times depend on arrival times in a piecewise affine manner. We find an efficient algorithm for calculating $d_i({\mathbf a})$. We then show that the total cost is a piecewise convex quadratic function but generally not convex, i.e. there is a large (but finite) number of polytopes in ${\mathbbm R}^N$ where within each polytope, it is a convex quadratic function of ${\mathbf a}$. This is a similar formulation to that of the piecewise-linear programming problem presented in \cite{VAN2010}, which is known to be NP-hard. The structure of the total cost yields an {\em exhaustive search} scheduling algorithm which terminates in finite time. 

We then put forward heuristics. The first heuristic, which we refer to as the {\em local search}, operates by solving a sequence of neighbouring quadratic problems until finding a local minimum with respect to the global optimization. The second heuristic performs a {\em global search} over one coordinate (arrival time of a single user), keeping other coordinates fixed. This is done in a provably efficient manner. In particular, we bound the number of steps in each coordinate search by a polynomial. It then repeats over other coordinates, cycling over all coordinates until no effective improvement in the objective function is possible. In case of smooth objectives, it is known that such Coordinate Pivot Iterations (CPI) schemes converge to local minima (see for example \cite{Bertsekas1999}, p272). Further, in certain special cases of non-smooth objectives, it is also known that CPI schemes converge to local minima (see for example \cite{Tseng2001}). But in our case, the non-separable piecewise structure of the objective often causes our heuristic to halt at a point that is not a local minimum. Nevertheless, the global search heuristic is fruitful when utilized in a {\em combined local-global search} heuristic.  This heuristic performs global searches with different initial points, each followed by a local search. We present numerical evidence, illustrating that it performs extremely well. Often finding the global optimum in very few steps.

The structure of the sequel is as follows. In Section~\ref{sec:model} we present the model and basic properties. In Section~\ref{sec:dynamics} we focus on arrival departure dynamics, showing a piecewise affine relationship between the arrival and departures times.  We give an efficient algorithm for calculating the departure times given arrival times or vice-versa. This also solves the scheduling problem for the special case $\gamma=0$. In Section~\ref{sec:optRegion} we characterise the constraints associated with quadratic programs which make up the piecewise quadratic cost. These are then used in the {\em exhaustive search} algorithm. We then present the {\em local search} algorithm and prove it always terminates at a local minimum (of the global objective). In Section~\ref{sec:CPIsection} we present our global search method based on CPI. We utilize the structure of the problem to obtain an efficient single coordinate search within the CPI. Then in Section~\ref{sec:combHeurist}, the local search and global searches are combined into a unified heuristic. We further illustrate the power of our heuristic through numerical examples. We conclude in Section~\ref{sec:Conclusion}. Some of the proofs are deferred to the appendix.

{\bf Notation:} We denote $x \wedge y$ and $x \vee y$ to be the minimum and maximum of $x$ and $y$, respectively. We define any summation with initial index larger than the final index to equal zero (e.g. $\sum_{i=2}^1 a_i=0$). Vectors are taken as columns and are denoted in bold.  ${\mathbf 1}\in\mathbbm{R}^N$ denotes a vector of $1$'s and $\mathbf{e_i}\in\mathbbm{R}^N$ denotes a vector of zeros in all but the $i$'th coordinate, which equals $1$. The indicator function is denoted by ${\mathbbm 1}$.

\section{Model}
\label{sec:model}

Our model assumes that there is a fixed user set $\mathcal{N}=\{1,\dots,N\}$ where the service requirement of each user, $\ell$, is the same and is set to $1$ without loss of generality (this can be accounted for by changing the units of $\beta$ and $\alpha$). Then the equations determining the relationship between the arrival times vector ${\mathbf a}=(a_1,\ldots,a_N)'$ and the departure times vector ${\mathbf d}=(d_1,\ldots,d_N)'$ are
\begin{equation}
\label{eq:dynamics}
1 = \int_{a_i}^{d_i} v\big(q(t)\big) dt,
\qquad
\mbox{where}
\qquad
q(t) = \sum_{j\in\mathcal{N}} {\mathbbm 1}\{t \in [a_j,d_j]\}.
\end{equation}
Using the linear slowdown rate function, \eqref{eq:ourSpeed}, the equations are represented as,
\begin{equation}
\label{eq:integratedForm}
1 = \int_{a_i}^{d_i} \beta  - \alpha \Big(\sum_{j\in\mathcal{N}}  {\mathbbm 1}\{t \in [a_j,d_j]\} -1 \Big) dt,
\qquad
i=1,\ldots,N.
\end{equation}
These $N$ equations can be treated as equations for the unknowns ${\mathbf d}$, given ${\mathbf a}$ or vice-versa.  We assume $N < {\beta}/{\alpha} +1$ so that it always holds that $v\big(q(t)\big) > 0$.

The cost incurred by user $i$ is,
\begin{equation}\label{eq:costUserShort}
c_i(a_i,d_i)= (d_i-d_i^*)^2 \,+ \gamma \,  (d_i - a_i),
\end{equation}
and the total cost function, which we seek to minimise, is
\begin{equation}
\label{eq:totalCost}
c( {\mathbf a}) 
= \sum_{i\in\mathcal{N}} c_i \big(a_i,d_i({\mathbf a}) \big).
\end{equation}
We assume (without loss of generality) that the ideal departure times, ${\mathbf d}^* = (d_1^*,\ldots,d_N^*)'$ are ordered, i.e. $d_1^*\le \ldots \le d_N^*$. 

\begin{remark}
For clarity of the exposition we choose the cost, \eqref{eq:costUserShort} to be as simplistic as possible. Practical straightforward generalizations to the cost and to the associated algorithms and heuristics are discussed in the conclusion of the paper. These include other convex penalty functions, ideal arrival times and a potentially different penalty for early and late departures. Our  algorithms, can all be adapted for such cost functions.
\end{remark}

We first have the following elementary lemmas:

\begin{lemma}\label{Lemma:departure_order}
Assume that the arrivals, ${\mathbf a}$, are ordered: $a_1 \le a_2 \le \ldots \le a_N$, then the departures, ${\mathbf d}$, follow the same order: $d_1 \le d_2 \le \ldots \le d_N$.
\end{lemma}

\begin{lemma}
\label{lemma:unique}
For any ${\mathbf a}$ there is a unique ${\mathbf d}$ and vice-versa. 
\end{lemma}
As a consequence of the assumed order of ${\mathbf d}^*$ and of the above lemma we assert that an optimal schedule can only be attained with an ordered ${\mathbf a}$ whose individual coordinates lie in a compact interval, as shown in the following lemma.

\begin{lemma}\label{Lemma:optimal_order}
An optimal arrival schedule satisfies $\underline{a} \le a_1 \le \ldots \le a_N \le \overline{a}$, where
\[
\underline{a} = d^*_1- \frac{N}{\beta- \alpha(N-1)},
\qquad
\overline{a} = d^*_N+\frac{N}{\beta- \alpha(N-1)}.
\]
\end{lemma}

We may thus define the search region for the optimal schedule:
\[
{\cal R} = \{ {\mathbf a} \in {\mathbbm R}^N ~:~ \underline{a} \le a_1 \le \ldots \le a_N \le \overline{a}   \},
\]
and take our scheduling problem to be $\min_{{\mathbf a} \in {\cal R}} ~c({\mathbf a})$.\\[1pt]

No strict condition on the joint order of $a_i$ and $d_i$ can be imposed except for the requirement that $a_i < d_i$ for any $i$ (the sojourn time of all users is strictly positive). We are thus motivated to define the following for $i\in\mathcal{N}$:
\begin{align}
\label{eq:KiDef}
k_i
&:= \max \big\{ k \in\mathcal{N} ~:~ a_k \le d_i \big\}
=
\min \big\{ k \in\mathcal{N} ~:~ a_{k+1} > d_i \big\},
\\
\label{eq:HiDef}
h_i
&:=\min 
\big\{ h \in\mathcal{N} ~:~ d_h \ge a_i\big\} 
=
 \max \big\{ h \in\mathcal{N} ~:~ d_{h-1} < a_i\big\}.
\end{align}
The variable $k_i$ specifies the interval $[a_{k_i},a_{k_i+1})$ in which $d_i$ resides. Similarly the variable $h_i$ specifies that $a_i$ lies in the interval $(d_{h_i-1},d_{h_i}]$. Note that we define $a_0,\,d_0:=-\infty$ and  $a_{N+1},\, d_{N+1}:=\infty$. The sequences $k_i$ and $h_i$ satisfy some basic properties: (i)~They are non-decreasing and are confined to the set $\cal{N}$. (ii)~From the fact that $a_i<d_i$ we have that $i \le k_i$.  (iii)~Since ${\mathbf d}$ is an ordered sequence and also $a_i < d_i$ we have $h_i \le i$. (iv)~We have $h_1=1$ and $k_N=N$. (v)~Each sequence determines the other:
\[
k_i = \max \big\{ k \in\mathcal{N}  : 
h_k \le i \big\},
\quad
\mbox{and}
\quad
h_i = \min \big\{ h \in\mathcal{N}  : 
k_h \ge i \big\}.
\]
Thus given either the sequence $k_i,~i\in\mathcal{N}$ or the sequence $h_i,~i\in\mathcal{N}$ or both, the ordering of the $2N$ tuple $(a_1,\ldots,a_N,d_1,\ldots,d_N)$ is fully specified as long as we require that $a_i$'s and $d_i$'s are ordered so as to be consistent with Lemmas \ref{Lemma:departure_order} and \ref{Lemma:optimal_order}. 

We denote the set of possible $\mathbf{k} = (k_1,\ldots,k_N)'$ by 
\begin{equation}
\mathcal{K}:=\left\lbrace\mathbf{k}\in{\cal N}^N \,:\,  k_N=N,\, k_i \le k_{j} \ \forall i \le j\right\rbrace.
\end{equation}
Similarly, we denote the set of possible ${\mathbf h} = (h_1,\ldots,h_N)'$ by ${\cal H}$. 
We have that,
\[
|{\cal K}| =|{\cal H}|= \frac{{2N \choose N}}{N+1}.
\]
This follows (for example) by observing that the elements of ${\cal K}$ correspond uniquely to lattice paths in the $N \times N$ grid from bottom-left to top-right with up and right movements without crossing the diagonal. The number of such elements is the $N$'th Catalan number, see for example p259 in \cite{koshy2009catalan}.

The following example illustrates the dynamics of the model (without optimization) and shows the role of ${\mathbf k}$, or alternatively ${\mathbf h}$, in summarizing the piecewise affine dynamics.

\begin{example}
Take $\beta=1/2$, $\alpha=1/6$ and $N=3$. This $3$ user system exhibits rates that are either $1/2, 1/3$ or $1/6$ depending on the number of users present. The free flow sojourn time is $1/\beta=2$. Assume $a_1=0$, $a_2=1$ and $a_3=3$. We now describe the dynamics of the system. See also Figure~\ref{fig:Example1}. 

\begin{figure}[h!]
\centering
\begin{tikzpicture}[xscale=2,yscale=0.15]
  \def\xmin{0}
  \def\xmax{5.75}
  \def\ymin{0}
  \def\ymax{35}
	\fill[color=lightgray]   (3,0) -- (3,30) -- (3.75,22.5) --  (5.25,0);	
	\fill[color=gray]   (1,0) -- (1,30) -- (2.5,15) -- (3,7.5) -- (3.75,0);	
	\fill[color=darkgray]   (0,0) -- (0,30)-- (1,15) -- (2.5,0);
	\draw[->, ultra thick] (\xmin,0) -- (\xmax,0) node[right] {$t$};
	\draw[->, ultra thick] (\xmin-0.1,0) -- (\xmin-0.1,\ymax) node[right] {work};
	
	\draw[->, black, ultra thick] (0,0)--(0,32) node[right] {$a_1=0.00 \ (h_1=1)$};
	\draw[->, black, ultra thick] (1,0)--(1,30) node[right] {$a_2=1.00 \ (h_2=1)$};
	\draw[->, black, ultra thick] (3,0)--(3,30) node[right] {$a_3=3.00 \ (h_3=2)$};	

	\draw[->, black, ultra thick] (2.5,0)--(2.5,-5) node[below] {$d_1=2.50 \ (k_1=2)$};
	\draw[->, black, ultra thick] (3.75,0)--(3.75,-8) node[below] {$d_2=3.75 \ (k_2=3)$};
	\draw[->, black, ultra thick] (5.25,0)--(5.25,-5) node[below] {$d_3=5.25 \ (k_3=3)$};		
	
\end{tikzpicture}
\caption{An illustration of the dynamics of a three user example. The shaded gray areas show the remaining work for each individual user. Work is depleted at rate $\frac{1}{2}$ when only one user is present and is depleted at the slower rate of $\frac{1}{3}$ when two users are present.
\label{fig:Example1}}
\end{figure}
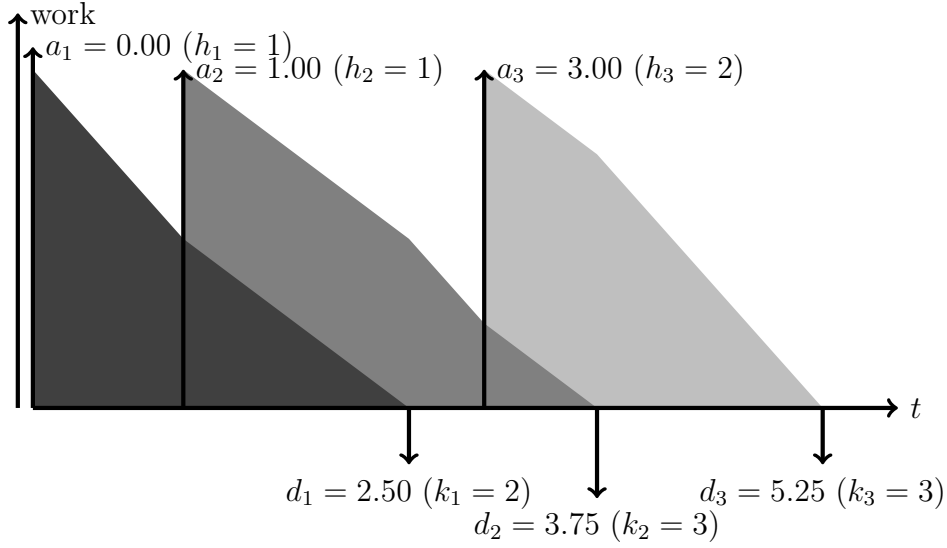

During the time interval $[0,1)$, $q(t)=1$ and the first user is being served at rate $1/2$. By time $t=1$ the remaining service required by that user is $1/2$. At time $t=1$, the number of users in the system, $q(t)$, grows to $2$ and the rate of service to each user is reduced to $1/3$. This means that without a further arrival causing further slowdown, user $1$ is due to leave at time $t=2.5$. Since $2.5<a_3$, this is indeed the case. At $t=2.5$, $q(t)$ changes from $2$ to $1$. By that time, the remaining service required by user $2$ is $1/2$. Then during the time interval $[2.5,3)$ user $2$ is served at rate $1/2$ reducing the remaining service of that user to $1/4$. At time $t=3$, user $3$ joins, increasing $q(t)$ back to $2$ and reducing the service rate again to $1/3$. User $2$ then leaves at time $t=3.75$ and as can be verified using the same types of simple calculations, user $3$ finally leaves at time $t=5.25$.

Observe that for this example, the order of events is:
\[
a_1 \le a_2 \le d_1 \le a_3 \le d_2 \le d_3.
\]
This then implies that for this schedule,
\[
k_1 = 2,~ k_2 = 3,~ k_3 = 3,
\quad
\mbox{and}
\quad
h_1 = 1,~ h_2 = 1,~ h_3 = 2.
\]
\end{example}

\section{Arrival Departure Dynamics}
\label{sec:dynamics}

We now investigate the relationship between arrivals and departures, induced by the linear slowdown dynamics. 
 \begin{proposition}\label{Lemma_departure_function}
Equation \eqref{eq:integratedForm} can be expressed as
\begin{equation}
\label{eq:dRecursion}
(\beta-\alpha(k_i-i))d_i-\alpha\sum_{j=h_i}^{i-1}d_j-(\beta-\alpha(i-h_i))a_i+\alpha\sum_{j=i+1}^{k_i}a_j=1 ,
\quad
i \in {\cal N},
\end{equation}
or alternatively,
\begin{equation}
\label{eq:daMatrix}
D\, {\mathbf d} - A \, {\mathbf a} = {\mathbf 1},
\end{equation}
with the matrices $A\in\mathbbm{R}^N$ and $D\in\mathbbm{R}^N$ defined as follows:
\[
A_{ij}:=
\left\{
	\begin{array}{ll}
		\beta-\alpha(i-h_i) \mbox{, } &  i=j, \\
		-\alpha \mbox{, } &  i+1\leq j\leq k_i , \\
		0\mbox{, } &  o.w.
	\end{array}
\right.
 \quad
D_{ij}:=
\left\{
	\begin{array}{ll}
		\beta-\alpha(k_i-i) \mbox{, } &  i=j, \\
		-\alpha \mbox{, } &  h_i\leq j\leq i-1,  \\
		0\mbox{, } &  o.w.
	\end{array}
\right. 
\]
\end{proposition}
\begin{proof}
We manipulate \eqref{eq:integratedForm} to get,
\begin{align*}
1 
&= (\beta+\alpha)(d_i-a_i) - \alpha \sum_{j=1}^N \int_{a_i}^{d_i}   {\mathbbm 1}\{t \in [a_j,d_j]\}dt\\ 
&=
(\beta+\alpha)(d_i-a_i)-\alpha\sum_{j=1}^N\left(d_i\wedge d_j-a_i\vee a_j\right)^{+} 
\\
&=
(\beta+\alpha)(d_i-a_i)
- \alpha\sum_{j=1}^{i-1}\left(d_i\wedge d_j-a_i\vee a_j\right)^{+}
-\alpha(d_i-a_i)
- \alpha \sum_{j=i+1}^N \left(d_i\wedge d_j-a_i\vee a_j\right)^{+}.
\\
&=
\beta(d_i-a_i)
- \alpha\sum_{j=1}^{i-1}\left(d_i\wedge d_j-a_i\vee a_j\right)^{+}
- \alpha \sum_{j=i+1}^N \left(d_i\wedge d_j-a_i\vee a_j\right)^{+}.
\end{align*}
where in the third step we have used the fact that $a_i < d_i$ for the term corresponding to $j=i$. We now use the fact that ${\mathbf a}$ and ${\mathbf d}$ are both ordered to get,
\begin{align*}
1
&=
\beta(d_i-a_i)
- \alpha\sum_{j=1}^{i-1}\left(d_j - a_i\right)^{+}
- \alpha \sum_{j=i+1}^N \left(d_i -  a_j \right)^{+}
\\
%
%
&= 
\beta(d_i-a_i) 
-\alpha \sum_{j=1}^{i-1} (d_j- a_i\wedge d_j)
-\alpha \sum_{j=i+1}^{N} (d_i - a_j\wedge d_i) 
\\
&= -\beta a_i+(\beta-\alpha(N-i))d_i
-\alpha\sum_{j=1}^{i-1}d_j
+\alpha\sum_{j=1}^{i-1}(a_i\wedge d_j)
+\alpha\sum_{j=i+1}^{N}(a_j\wedge d_i).
\end{align*}
Now the summations $\sum_{j=1}^{i-1}(a_i\wedge d_j)$ and $\sum_{j=i+1}^{N}(a_j\wedge d_i)$ can be broken up as follows:
\begin{align*}
\sum_{j=1}^{i-1}(a_i\wedge d_j) 
&= 
\sum_{j=1}^{i-1} {\mathbbm 1}\{d_j < a_i\} d_j + 
\sum_{j=1}^{i-1} {\mathbbm 1}\{d_j \ge a_i\} a_i\\
&=
\sum_{j=1}^{h_i-1} d_j + 
\sum_{j=h_i}^{i-1} a_i
=
\sum_{j=1}^{h_i-1} d_j 
\,+ \,
(i-h_i) a_i,\\
\sum_{j=i+1}^{N}(a_j\wedge d_i)
&= \sum_{j=i+1}^N {\mathbbm 1} \{a_j > d_i\} d_i+ 
\sum_{j=i+1}^N {\mathbbm 1} \{a_j \le d_i\} a_j\\
&= 
\sum_{j=k_i+1}^N d_i
+
\sum_{j=i+1}^{k_i} a_j
=
(N-k_i) d_i + \sum_{j=i+1}^{k_i} a_j.
\end{align*}
Combining the above we obtain:
\begin{align*}
1 
&= -(\beta - \alpha(i-h_i))a_i + (\beta-\alpha(k_i-i))d_i - \alpha \Big(
\sum_{j=1}^{i-1} d_j - \sum_{j=1}^{h_i-1} d_j
\Big)
+\alpha \sum_{j=i+1}^{k_i} a_j
\\
&= -(\beta - \alpha(i-h_i))a_i + (\beta-\alpha(k_i-i))d_i - \alpha \Big(
\sum_{j=h_i}^{i-1} d_j - \sum_{j=i+1}^{k_i} a_j
\Big).
\end{align*}
Rearranging we obtain \eqref{eq:dRecursion}.
\qed\end{proof}

\noindent
The following observations are a consequence of Proposition \ref{Lemma_departure_function}:

\begin{enumerate}
\item Consider some user $i$ arriving at time $a_i$ to an empty system, and departing at time $d_i$ to leave an empty system. In this case there are no other users effecting his sojourn time or rate. For such a user $k_i=h_i=i$. In this case \eqref{eq:dRecursion} implies that $d_i = a_i + 1/\beta$ as expected.
\item The matrices $A$ and $D$ are lower and upper triangular, respectively, with a non-zero diagonal, and are therefore both non-singular. 
\item 
For the special cases $i=1$ and $i=N$ (using the fact $h_1=1$ and $k_N=N$):
\[
\big( \beta - \alpha(k_1 - 1) \big) d_1 - \beta\, a_1 + \alpha \sum_{j=2}^{k_1} a_j = 1,
\quad
\mbox{and}
\quad
\beta \, d_N - \alpha \sum_{j=h_N}^{N-1} d_j - \big( \beta - \alpha(N-h_N) \big) a_N = 1.
\]
I.e.,
\[
d_1 =  \frac{1+\beta a_1-\alpha\sum_{j=2}^{k_1}a_j}{\beta-\alpha(k_1-1)},
\qquad
a_N = \frac{\beta d_N - \alpha \sum_{j=h_N}^{N-1} d_j - 1 }{\beta - \alpha(N - h_N)}.
\]
%
%
\end{enumerate}

The above structure suggests iterative algorithms for either determining ${\mathbf a}$ based on ${\mathbf d}$ or vice-versa. In both cases, ${\mathbf k}$ and ${\mathbf h}$ are found as bi-products. As an aid to describing these algorithms, define for $i, k, h \in {\cal N}$ and for a given ${\mathbf a}$ (respectively ${\mathbf d}$), the functions $\tilde{d}_{i,k,h}(\cdot \,|\, \mathbf{a}), \tilde{a}_{i,k,h}(\cdot \,|\, \mathbf{d}): {\mathbb R}^N \to {\mathbb R}$ as follows,
\begin{align*}
\tilde{d}_{i,k,h}\Big(\tilde{{\mathbf d}} \,\Big|\, \mathbf{a} 
\Big) 
&:=
\frac{
1 + \big(\beta - \alpha(i-{h})\big)a_i + \alpha \Big(
\sum_{j={h}}^{i-1} \tilde{d}_j - \sum_{j=i+1}^{{k}} a_j
\Big)
}{\beta-\alpha({k}-i)},
\\
\tilde{a}_{i,k,h}\Big(\tilde{{\mathbf a}} \,\Big|\, \mathbf{d}
\Big) 
&:=
\frac{
 \big(\beta - \alpha(k-i)\big)d_i - \alpha \Big(
\sum_{j={h}}^{i-1} {d}_j - \sum_{j=i+1}^{{k}} \tilde{a}_j
\Big) -1
}{\beta-\alpha(i-h)}.
\end{align*}
Observe that in the evaluation of these functions, the arguments, $\tilde{{\mathbf d}}$ or $\tilde{\mathbf a}$ are only utilized for the coordinates indexed $h,\ldots,i-1$ or $i+1,\ldots,k$ respectively (if $i=1$ or respectively $i=N$ these index lists are empty). Further observe that stated in terms of $\tilde{d}(\cdot)$ or $\tilde{a}(\cdot)$ and given ${\mathbf k} \in {\cal K}$ and ${\mathbf h} \in {\cal H}$,  equations  \eqref{eq:dRecursion} can be represented as,
\[
d_i = \tilde{d}_{i,k_i,h_i}\Big(\big(d_1,\ldots,d_{N}\big)' \, | \, \big(a_1,\ldots,a_{N}\big)'\Big),
\qquad
i\in\mathcal{N},
\]
or alternatively,
\[
a_i = \tilde{a}_{i,k_i,h_i}\Big(\big(a_1,\ldots,a_{N}\big)' \, | \, \big(d_1,\ldots,d_{N}\big)'\Big),
\qquad
i\in\mathcal{N}.
\]
Given the above we have two (dual) algorithms for determining the network dynamics. Algorithm~1a finds the departure times based on arrival times. Algorithm~1b finds the arrival times given the departure times. 
\begin{proposition}
\label{prop:alg1_complexity}
Algorithm~1a finds the unique solution ${\mathbf d}$ to equations \eqref{eq:integratedForm}, given~${\mathbf a}$. Similarly Algorithm~1b finds a unique solution ${\mathbf a}$ to the equations, given ${\mathbf d}$. Both algorithms require at most $2N$ steps in each of which \eqref{eq:dRecursion} is evaluated.
\end{proposition}

\begin{algorithm}[h]
\renewcommand{\thealgorithm}{}
\caption{\textbf{1a}: Determination of network dynamics with given arrival times}
\label{algo:NetworkDynamics}
\textbf{Input}: $\mathbf{a} \in {\mathbb R}^N$ such that $a_1\le a_2 \le \ldots \le a_N$ \\
\textbf{Output}: $\mathbf{d}=(d_1,...,d_N)$, $\mathbf{k}=(k_1,...,k_N)$ and $\mathbf{h}=(h_1,...,h_N)$
\begin{algorithmic}
\State init $\mathbf{k}=\mathbf{h}=(1,2,3,\dots,N)$
\State init $\mathbf{d}=\emptyset$
\For{ $i=1,\ldots,N$}  
	\State set $k=i\vee k_{i-1}$  (taking $k_0 := 1$)
		\State compute $\tilde{d}_i(k,h_i,\mathbf{d} \, | \, {\mathbf a})$
	\While{$\tilde{d}_i(k,h,\mathbf{d} \, | \, {\mathbf a}) \le a_{k+1}$}
		\State increment $k$
		\State compute $\tilde{d}_i(k,h_i,\mathbf{d} \, | \, {\mathbf a})$		
	\EndWhile
	\State set $k_i=k$	
	\State set $d_i=\tilde{d}_i(k,h_i,\mathbf{d}\, | \, {\mathbf a})$
	\State set $h_{i+1}=
\max \big\{ h\in \{1,\ldots,i+1\}  : 
k_h \ge i+1 \big\}$
\EndFor	
\State \textbf{return} {$(\mathbf{d},\mathbf{k},\mathbf{h})$}
\end{algorithmic}
\end{algorithm}

\begin{algorithm}
\renewcommand{\thealgorithm}{}
\caption{\textbf{1b}: Determination of network dynamics with given departure times}\label{algo:NetworkDynamicsBasedD}
\textbf{Input}: $\mathbf{d} \in {\mathbb R}^N$ such that $d_1\le d_2 \le \ldots \le d_N$ \\
\textbf{Output}: $\mathbf{a}=(a_1,...,a_N)$, $\mathbf{k}=(k_1,...,k_N)$ and $\mathbf{h}=(h_1,...,h_N)$
\begin{algorithmic}
\State init $\mathbf{k}=\mathbf{h}=(1,2,3,\dots,N)$
\State init $\mathbf{d}=\emptyset$
\For{ $i=N,\ldots,1$}  
	\State set $h=i \wedge h_{i+1}$  (taking $h_{N+1} := N$)
		\State compute $\tilde{a}_i(k_i,h,\mathbf{a} \, | \, {\mathbf d})$
	\While{$\tilde{a}_i(k_i,h,\mathbf{a} \, | \, {\mathbf d}) \ge d_{h-1}$}
		\State decrement $h$
		\State compute $\tilde{a}_i(k_i,h,\mathbf{a} \, | \, {\mathbf d})$		
	\EndWhile
	\State set $h_i=h$	
	\State set $a_i=\tilde{a}_i(k_i,h,\mathbf{a} \, | \, {\mathbf d})$
	\State set $k_{i-1}= \min \big\{ k\in \{i-1,\ldots,N\}  : 
h_k \le i-1 \big\}$
\EndFor	
\State \textbf{return} {$(\mathbf{a},\mathbf{k},\mathbf{h})$}
\end{algorithmic}
\end{algorithm}
%

\subsection{Optimizing for Extreme Cases of $\gamma$}
\label{sec:polGamma}

As described in the introduction, optimizing \eqref{eq:totalCost} when $\gamma=0$ or $\gamma \approx \infty$ can be done efficiently. For the case $\gamma=0$, all that is needed is to schedule arrivals so that each departure time, $d_i$ is exactly at $d_i^*$.  This achieves zero costs. 
Such a schedule is simply obtained by running Algorithm~1b with input ${\mathbf d}= {\mathbf d}^*$. This immediately leads to the following corollary of Proposition~\ref{prop:alg1_complexity}:
\begin{corollary}
\label{thm:efficientGammaZero}
For the special case $\gamma=0$ there is an efficient polynomial time algorithm that finds the unique optimal schedule, ${\mathbf a}^0$, achieving $c({\mathbf a}^0)=0$.
\end{corollary}
For the case of large $\gamma$ it is sensible to consider a classic schedule where users do not overlap:
\begin{equation}
\label{eq:NoOverlapConst}
a_i + \frac{1}{\beta} = d_i \le a_{i+1},
\quad
i=1,\ldots,N-1.
\end{equation}
This poses the problem as a classic single machine scheduling problem with due dates (see for example \cite{baker1990sequencing} or \cite{sen1984state}). This implies that the total costs due to sojourn times is at the minimal possible value $\gamma N / \beta$ and the costs due to deviations from ideal departure times is,
\[
\sum_{i \in {\cal N}} (a_i + 1/\beta - d_i^*)^2.
\]
For any finite $\gamma$ this does not necessarily minimize \eqref{eq:totalCost}, but as $\gamma \to \infty$ it is a sensible approximation. I.e. for large $\gamma$ the optimal schedule is approximated by the solution of the following convex quadratic program:
\begin{equation}\label{bigGammaQP}
\begin{aligned}
& \underset{(a_1,\ldots,a_N)' \in {\mathbb R}^N}{\text{min}}
& & \sum_{i=1}^N (a_i + 1/\beta - d_i^*)^2 \\
& \text{s.t.}
& & a_i - a_{i+1} \le -\frac{1}{\beta}, \quad i=1,\ldots,N-1. \\
\end{aligned}
\end{equation}

As for the case $\gamma=0$, the above quadratic program can be efficiently solved using any standard convex quadratic programming method. Denote the optimizer by ${\mathbf a}^\infty$.

\subsection{A Linear Approximation}
\label{sec:linApprox}

Having the schedules  ${\mathbf a}^0$ and ${\mathbf a}^\infty$ for the cases $\gamma=0$ and $\gamma=\infty$ respectively, we are motivated to suggest a set of potential (initial) guesses for the optimal schedule for arbitrary $\gamma$.  Let $M\ge1$ be some integer specifying the number of initial guesses. Then the set of initial guesses lie on the segment interpolating ${\mathbf a}^0$ and ${\mathbf a}^\infty$:
\begin{equation}
\label{eq:3488}
{\cal A} =\Big\{ {\mathbf a}^0  \,\frac{m}{M-1} +  {\mathbf a}^\infty \, \Big(1 - \frac{m}{M-1} \Big) 
~:~
m=0,\ldots,M-1
\Big\},\end{equation}
when $M\ge 2$ or equals $\{{\mathbf a}^0\}$ if $M=1$.
We shall use the $M$ points of ${\cal A}$ as initial guess points for the optimization heuristics that we present in the sequel. This is a sensible choice since every set of due dates $d_1^*,\ldots,d_N^*$ exhibits some contour in ${\cal R}$, parametrized by $\gamma$, corresponding to the optimal schedules (for each $\gamma$). The end points of this contour are ${\mathbf a}^0$ and ${\mathbf a}^\infty$ which we can efficiently find. Thus for $\alpha \in [0,1]$, the points ${\mathbf a}^0  \,\alpha +  {\mathbf a}^\infty \, (1-\alpha)$ constitute a linear approximation of this contour. In cases where the contour is almost not curved we have that the optimal value lies very near to the linear approximation. In other cases, this is simply a set of initial guesses, yet possibly a sensible one. {Note that the values of $M$ do not need to be excessively large because initial points that are close are likely to yield the same local solutions. The numerical analysis of Section \ref{sec:combHeurist} reinforces this observation.}

\section{Piecewise Quadratic Formulation}
\label{sec:optRegion}

Our key observation in this section is that the search region ${\cal R}$ can be partitioned into polytopes indexed by ${\mathbf k} \in {\cal K}$, where over each such polytope, the objective is of a convex quadratic form. This yields $|{\cal K}|$ convex quadratic problems, each of which (individually) can be solved efficiently.  An immediate exhaustive-search algorithm is then to solve all of the problems so as to find the minimising one.  This yields a finite-time exact solution and is a sensible choice for small $N$ (e.g. $N\le 15$). But since,
\[
| {\cal K}| \sim \frac{4^N}{N^{3/2} \sqrt{\pi}},
\]   
solving all convex problems is not a viable method for non-small $N$. We thus also specify a {\em local-search} algorithm which searches elements of ${\cal K}$ by moving across neighbouring polytopes until finding a local optimum.  

The following proposition is key:
\begin{proposition}
The region ${\cal R}$ can be partitioned into polytopes indexed by ${\mathbf k} \in {\cal K}$, and denoted 
\begin{equation*}
\mathbbm{P}_\mathbf{k}:=\left\lbrace \mathbf{a}\in{\cal R} \,:\, a_{k_i}\leq [\Theta_{\mathbf k} \mathbf{a}+{\boldsymbol \eta}_{\mathbf k} ]_i\leq a_{k_i+1},~ i \in {\cal N} \right\rbrace ,
\end{equation*}
where  $\Theta_{\mathbf k} = D^{-1} A$ and $\eta_{\mathbf k} = D^{-1} {\mathbf 1}$ with $A$ and $D$ based on ${\mathbf k}$ are specified by Proposition~\ref{Lemma_departure_function}. Then for ${\mathbf a} \in \mathbbm{P}^\mathbf{k}$ the objective function is convex and is given by,
\[
c_{\mathbf k}({\mathbf a}) =  \mathbf{a}'Q_{\mathbf k} \mathbf{a}+\mathbf{b}_{\mathbf k} \, \mathbf{a}+{\tilde{b}}_{\mathbf k},
\]
with,
\[
\begin{split}
Q_{\mathbf k} &= \Theta'_{\mathbf k} \, \Theta_{\mathbf k} , \\
{\mathbf b}_{\mathbf k} &= 2 \big({\boldsymbol \eta}_{\mathbf k} - {\mathbf d}^*)'\Theta_{\mathbf k} + \gamma {\mathbf 1}'(\Theta_{\mathbf k} -I), \\
{\tilde b}_{\mathbf k} &= ({\boldsymbol \eta}_{\mathbf k} -{\mathbf d}^{*})' ({\boldsymbol \eta}_{\mathbf k} -{\mathbf d}^{*})+{\boldsymbol \eta}_{\mathbf k}.
\end{split}
\] 
 \end{proposition}
\begin{proof}
The results of Proposition~\ref{Lemma_departure_function} show that every ${\mathbf k} \in {\cal K}$ specifies matrices $D$ and $A$ such that, ${\mathbf d} = D^{-1} A ~ {\mathbf a} + D^{-1} {\mathbf 1} = \Theta_{\mathbf k} {\mathbf a} + {\boldsymbol \eta}_{\mathbf k}$. This holds with constant $\Theta_{\mathbf k}$ and ${\boldsymbol \eta}_{\mathbf k}$ for all ${\mathbf a}$ and ${\mathbf d}$ for which ${\mathbf k}$ as defined in \eqref{eq:KiDef} is fixed. The polytope $\mathbbm{P}_\mathbf{k}$ specifies this exactly by describing the set of arrival points for which the specific ordering of departures within arrivals is given by ${\mathbf k}$.

Since for all ${\mathbf a} \in\mathbbm{P}_\mathbf{k}$ the affine relationship between ${\mathbf a}$ and ${\mathbf d}$ holds with the same $\Theta_{\mathbf k}$ and ${\boldsymbol \eta}_{\mathbf k}$  the cost, \eqref{eq:totalCost}, can be explicitly represented in terms of ${\mathbf a}$:
\begin{align*}
c({\mathbf a}) &= \sum_{i\in\mathcal{N}} (d_i - d_i^*)^2 + \gamma(d_i-a_i)
\\
&= ({\mathbf d}-{\mathbf d}^*)' \,({\mathbf d}-{\mathbf d}^*) + \gamma {\mathbf 1}' \, ({\mathbf d} - {\mathbf a}) \\
&=\big({\mathbf a}' \Theta'_{\mathbf k} + ({\boldsymbol \eta}_{\mathbf k} -{\mathbf d}^{*})') \,( \Theta_{\mathbf k} {\mathbf a} + {\boldsymbol \eta}_{\mathbf k}-{\mathbf d}^*) + \gamma {\mathbf 1}'  \,( \Theta_{\mathbf k} {\mathbf a} + {\boldsymbol \eta}_{\mathbf k} - {\mathbf a}) \\
& = {\mathbf a}' \Theta'_{\mathbf k} \Theta_{\mathbf k} {\mathbf a} + \big(2 \big({\boldsymbol \eta}_{\mathbf k} - {\mathbf d}^*)'\Theta_{\mathbf k} + \gamma {\mathbf 1}'(\Theta_{\mathbf k} -I) \big)  {\mathbf a}
+ ({\boldsymbol \eta}_{\mathbf k} -{\mathbf d}^{*})' ({\boldsymbol \eta}_{\mathbf k} -{\mathbf d}^{*})+\gamma {\mathbf 1}'{\boldsymbol \eta}_{\mathbf k}.
\end{align*}
This yields $Q_{\mathbf k}$, ${\mathbf b}_{\mathbf k}$ and the constant term, $\tilde{b}_{\mathbf k}$. Finally, since $Q_{\mathbf k}$ is a Gram matrix, it is positive semi-definite. Hence the objective is convex.
\qed\end{proof}

\subsection{Exhaustive Search}
We are now faced with a family of convex quadratic programs. For each ${\mathbf k} \in {\cal K}$, denote $c_{\mathbf k}(\cdot)$ to be the cost associated with ${\mathbf k}$ then,
\begin{equation}\label{eq:QP(k)}
QP({\mathbf k}): \qquad \min_{{\mathbf a} \in\mathbbm{P}_\mathbf{k}} c_{\mathbf k}({\mathbf a}).	
\end{equation}

Note that while the constant term $ \tilde{b}_{\mathbf k}$ is not required for finding the solution of $QP({\mathbf k})$, it is needed for comparing the outcomes of the quadratic programs associated with different elements of ${\cal K}$. Indeed the most basic use of $QP({\mathbf k})$ is for an exhaustive search algorithm which finds the global optimal schedule in finite time. This is summarised in Algorithm~2.

\begin{algorithm}
\renewcommand{\thealgorithm}{}
\caption{\textbf{2}: Exhaustive search for global optimum}
\textbf{Input}: Model parameters only 
($N,\alpha,\beta,\mathbf{d}^*$ and $\gamma$) \\
\textbf{Output}: $\mathbf{a}^*$ (global optimum)
\begin{algorithmic}
\State init $m^*=\infty$
\For{$\mathbf{k}\in\mathcal{K}$}
	\State solve  $QP({\mathbf k})$ with optimiser ${\mathbf a}$ and optimum $m$
	\If{$m<m^*$}
		\State set $\mathbf{a}^*=\mathbf{a}$
		\State set $m^*=m$
	\EndIf
\EndFor
\State \textbf{return} {$(\mathbf{a}^*,\ m^*)$}
\end{algorithmic}
\end{algorithm}

The virtue of Algorithm~2 is that it finds the optimal schedule in finite time. But this is done by solving an exponential (in $N$) number of convex $QP(\cdot)$ problems, so for non-small $N$ it is not a sensible algorithm. Hence we now introduce a search heuristic.
\subsection{Neighbour Search}
In this section we introduce a heuristic search aimed at finding a local minimum by searching on neighbouring regions. The search procedure solves the QP \eqref{eq:QP(k)} over neighbouring elements of $\mathcal{K}$ by changing a single coordinate of $\mathbf{k}$ at a time. We prove that this procedure converges to a local minimum; yet this may possibly take an exponential number of steps in the worst case.

Given a solution ${\boldsymbol a}$ of $QP({\boldsymbol k})$ we define the following two sets of indices:
\begin{align*}
{\cal I}_1(\mathbf{a},\, \mathbf{k})&:=\big\lbrace j\in {\cal N}
\,:\,
[\Theta_{\mathbf k} \mathbf{a}+{\boldsymbol \eta}_{\mathbf k} ]_j= a_{k_j+1} \big\rbrace,
\\
{\cal I}_2(\mathbf{a},\, \mathbf{k})&:=\big\lbrace j\in {\cal N} \,:\,
a_{k_{j}} = [\Theta_{\mathbf k} \mathbf{a}+{\boldsymbol \eta}_{\mathbf k} ]_{j}  \big\rbrace.
\end{align*}

Noting that $d_i =[\Theta_{\mathbf k} \mathbf{a}+{\boldsymbol \eta}_{\mathbf k} ]_i$, and recalling that $k_i$ is index of the maximal arrival time that is less than or equal to $d_i$ we have that if $i \in {\cal I}_1(\mathbf{a},\, \mathbf{k})$ then the optimal solution of $QP({\boldsymbol k})$ exhibits $d_i =a_{k_i+1}$ as an active constraint. Hence a neighbouring region to the constraint set $\mathbbm{P}_\mathbf{k}$ is ${\mathbbm P}_\mathbf{k^{(i)}}$ where $\mathbf{k}^{(i)}=\mathbf{k}$ on all coordinates except for $i$ where it is equal to $k_i+1$.  Similarly if $i \in {\cal I}_2(\mathbf{a},\, \mathbf{k})$ then $a_{k_i} =d_i$ as an active constraint. In this case, $\mathbf{k}^{(i)}$ is set to equal $\mathbf{k}$ on all co-ordinates except for $i$ where it is set to equal $k_i-1$.  Thus for every element of ${\cal I}_1(\mathbf{a},\, \mathbf{k})$ and ${\cal I}_2(\mathbf{a},\, \mathbf{k})$ we have a well defined neighbouring region.  Defining now the sets of neighbouring regions to $\mathbbm{P}_\mathbf{k}$ by
 \[
 {\cal K}_\ell\big({\cal I}_\ell(\mathbf{a},\, \mathbf{k})\big) 
 :=
 \big\{ {\mathbf k}^{(i)} ~:~ i \in {\cal I}_\ell(\mathbf{a},\, \mathbf{k}) \big\}
 , \quad \ell=1,2,
 \]
we have the following local search algorithm:

\begin{algorithm}[H]
\renewcommand{\thealgorithm}{}
\caption{\textbf{3}: Neighbour search for local optimum (local search)}
\label{algo:neighbour}
\textbf{Input}: $\mathbf{k}$\\
\textbf{Output}: $\mathbf{a}^*$ and $m^*$
\begin{algorithmic}
\State  solve  $QP({\mathbf k})$ with optimiser ${\mathbf a}$ and optimum $m$
\State init $m^*=m$
\State init $\mathbf{a}^*=\mathbf{a}$
	\For{$i \in \ {\cal K}_1\big({\cal I}_1(\mathbf{a}, \mathbf{k})\big)$}
		\State solve  $QP({\mathbf k^{(i)}})$ with optimiser ${\mathbf a}$ and optimum $m$
		\State {\bf if} $m<m^*$   {\bf then} restart algorithm with $\mathbf{k}=\mathbf{k^{(i)}}$
		\EndFor
	\For{$i \in \ {\cal K}_2\big({\cal I}_2(\mathbf{a}, \mathbf{k})\big)$}
		\State solve  $QP({\mathbf k^{(i)}})$ with optimiser ${\mathbf a}$ and optimum $m$
			\State {\bf if} $m<m^*$   {\bf then} restart algorithm with $\mathbf{k}=\mathbf{k^{(i)}}$
	\EndFor
\State \textbf{return} {$(\mathbf{a}^*,\ m^*)$}
\end{algorithmic}
\end{algorithm}

\begin{proposition}
Algorithm 3 converges to a local minimum for any initial vector $\mathbf{k}$.
\end{proposition}
\begin{proof}
Every step of the algorithm can only improve the objective function, since $m<m^*$ is the condition for the change of $\mathbf{k}$, hence the algorithm cannot go back to a region which it has already visited. Furthermore, there is a finite number of regions which means the algorithm terminates in a finite number of steps. If for some $\mathbf{a}$ which is the solution of $QP(\mathbf{k})$ there are no improvements in any of the neighbouring regions the algorithm stops at a local minimum. This can be either due to no active constraints to $QP(\mathbf{k})$ (an interior point) or due to the fact that the neighbouring quadratic programs do not improve on the solution of $QP(\mathbf{k})$.
\qed
\end{proof}

\section{Global Search Over Single Coordinates}
\label{sec:CPIsection}

In this section we put forward Algorithms 4 and 5 that together form a coordinate pivot iteration procedure. We first describe how the dynamics presented in Sections \ref{sec:model} and \ref{sec:optRegion} can be used to find a global minimum with respect to a single coordinate $r \in {\cal N}$ (user) when all other coordinates are fixed. We call this procedure a {\em global search over a single coordinate} $r$. 

The computational complexity of such a procedure is shown to be at most $O(N^5)$. We then utilise this procedure to define a coordinate pivot iteration algorithm, that performs optimization cycles on all of the coordinates until no improvement can be made.

To understand the main idea consider Figure~\ref {fig:departure_arrival_xx}~(a). This figure corresponds to an example with $N=4$, $\alpha=1.5$ and $\beta=5$. Here the arrival times $a_2,\, a_3,\, a_4$ are fixed at $(0.05, 0.15, 0.45)$ and the arrival time of user $r=1$ (denoted also $x$) is allowed to vary. The (horizontal) blue dotted lines denote the fixed arrival times $a_2,\, a_3,\, a_4$. The thin blue curves correspond to the departure times $d_2, \, d_3, \, d_4$. The thick green dotted and solid curves correspond to the arrival and departure time of user $1$ respectively. When $x$ is small enough or large enough, it is seen that user $1$ does not affect the other users. But otherwise, user $1$ interacts with the other users and potentially modifies their departure times. 

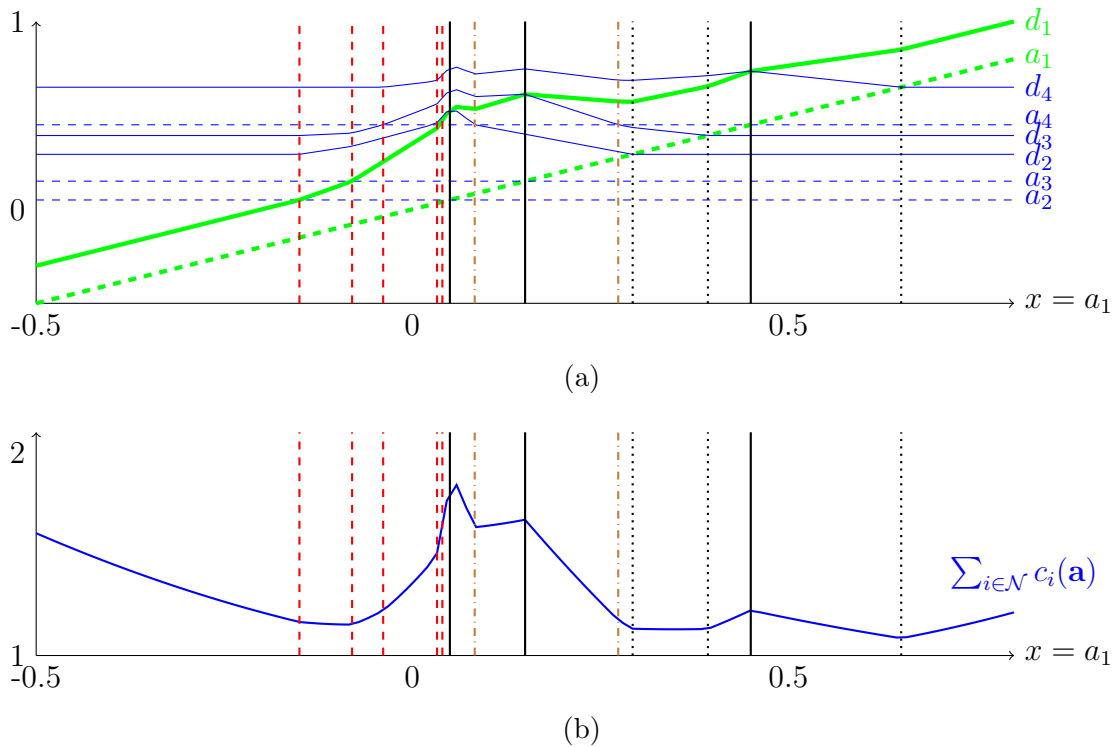
\begin{figure}[H]
\centering
\begin{subfigure}{\columnwidth}
\begin{tikzpicture}[xscale=10,yscale=2.5]
  \def\xmin{-0.5}
  \def\xmax{0.8}
  \def\ymin{-0.5}
  \def\ymax{1}
    \draw[->] (\xmin,\ymin) -- (\xmax,\ymin) node[right] {$x=a_1$} ;
    \draw[->] (\xmin,\ymin) -- (\xmin,\ymax); 
    \foreach \x in {-0.5,0,0.5}
    \node at (\x,\ymin) [below] {\x};
    \foreach \y in {0,1}
    \node at (\xmin,\y) [left] {\y};
    
    \draw[smooth,dashed,green,ultra thick] (-0.5,-0.5)--	(-0.1621,-0.1621)--	(-0.1491,-0.1491)--	(-0.1361,-0.1361)--	(-0.1231,-0.1231)--	(-0.1101,-0.1101)--	(-0.0971,-0.0971)--	(-0.0841,-0.0841)--	(-0.0711,-0.0711)--	(-0.0581,-0.0581)--	(-0.0451,-0.0451)--	(-0.0321,-0.0321)--	(-0.0191,-0.0191)--	(-0.0061,-0.0061)--	(0.0069,0.0069)--	(0.0199,0.0199)--	(0.0329,0.0329)--	(0.0459,0.0459)--	(0.0589,0.0589)--	(0.0719,0.0719)--	(0.0849,0.0849)--	(0.0979,0.0979)--	(0.1109,0.1109)--	(0.1239,0.1239)--	(0.1369,0.1369)--	(0.1499,0.1499)--	(0.1629,0.1629)--	(0.1759,0.1759)--	(0.1889,0.1889)--	(0.2019,0.2019)--	(0.2149,0.2149)--	(0.2279,0.2279)--	(0.2409,0.2409)--	(0.2539,0.2539)--	(0.2669,0.2669)--	(0.2799,0.2799)--	(0.2929,0.2929)--	(0.3059,0.3059)--	(0.3189,0.3189)--	(0.3319,0.3319)--	(0.3449,0.3449)--	(0.3579,0.3579)--	(0.3709,0.3709)--	(0.3839,0.3839)--	(0.3969,0.3969)--	(0.4099,0.4099)--	(0.4229,0.4229)--	(0.4359,0.4359)--	(0.4489,0.4489)--	(0.4619,0.4619)--	(0.4749,0.4749)--	(0.4879,0.4879)--	(0.5009,0.5009)--	(0.5139,0.5139)--	(0.5269,0.5269)--	(0.5399,0.5399)--	(0.5529,0.5529)--	(0.5659,0.5659)--	(0.5789,0.5789)--	(0.5919,0.5919)--	(0.6049,0.6049)--	(0.6179,0.6179)--	(0.6309,0.6309)--	(0.6439,0.6439)--	(0.6569,0.6569)--	(0.6699,0.6699)--	(0.6829,0.6829)--	(0.6959,0.6959)--	(0.7089,0.7089)--	(0.7219,0.7219)--	(0.7349,0.7349)--	(0.7479,0.7479)--	(0.7609,0.7609)--	(0.7739,0.7739)--	(0.7869,0.7869)--	(0.7999,0.7999);
    
     \draw[smooth,dashed,blue] (-0.5,0.05) -- (0.8,0.05);
     
     \draw[smooth,dashed,blue] (-0.5,0.15) -- (0.8,0.15);
     
     \draw[smooth,dashed,blue] (-0.5,0.45) -- (0.8,0.45);

    \draw[smooth,green,ultra thick] (-0.5,-0.3)--	(-0.1621,0.0379)--	(-0.1491,0.05128571)--	(-0.1361,0.06985714)--	(-0.1231,0.08842857)--	(-0.1101,0.107)--	(-0.0971,0.12557143)--	(-0.0841,0.14414286)--	(-0.0711,0.17225)--	(-0.0581,0.20475)--	(-0.0451,0.23725)--	(-0.0321,0.26975)--	(-0.0191,0.30225)--	(-0.0061,0.33475)--	(0.0069,0.36725)--	(0.0199,0.39975)--	(0.0329,0.43225)--	(0.0459,0.509)--	(0.0589,0.54555)--	(0.0719,0.53905)--	(0.0849,0.53519375)--	(0.0979,0.55063125)--	(0.1109,0.56606875)--	(0.1239,0.58150625)--	(0.1369,0.59694375)--	(0.1499,0.61238125)--	(0.1629,0.60846875)--	(0.1759,0.60440625)--	(0.1889,0.60034375)--	(0.2019,0.59628125)--	(0.2149,0.59221875)--	(0.2279,0.58815625)--	(0.2409,0.58409375)--	(0.2539,0.58003125)--	(0.2669,0.57596875)--	(0.2799,0.57337449)--	(0.2929,0.57248397)--	(0.3059,0.58309621)--	(0.3189,0.59370845)--	(0.3319,0.6043207)--	(0.3449,0.61493294)--	(0.3579,0.62554519)--	(0.3709,0.63615743)--	(0.3839,0.64676968)--	(0.3969,0.65985714)--	(0.4099,0.67842857)--	(0.4229,0.697)--	(0.4359,0.71557143)--	(0.4489,0.73414286)--	(0.4619,0.74251429)--	(0.4749,0.74994286)--	(0.4879,0.75737143)--	(0.5009,0.7648)--	(0.5139,0.77222857)--	(0.5269,0.77965714)--	(0.5399,0.78708571)--	(0.5529,0.79451429)--	(0.5659,0.80194286)--	(0.5789,0.80937143)--	(0.5919,0.8168)--	(0.6049,0.82422857)--	(0.6179,0.83165714)--	(0.6309,0.83908571)--	(0.6439,0.84651429)--	(0.6569,0.8569)--	(0.6699,0.8699)--	(0.6829,0.8829)--	(0.6959,0.8959)--	(0.7089,0.9089)--	(0.7219,0.9219)--	(0.7349,0.9349)--	(0.7479,0.9479)--	(0.7609,0.9609)--	(0.7739,0.9739)--	(0.7869,0.9869)--	(0.7999,0.9999);

\draw[smooth,blue] (-0.5,0.2928571)--	(-0.1621,0.2928571)--	(-0.1491,0.2934082)--	(-0.1361,0.3013673)--	(-0.1231,0.3093265)--	(-0.1101,0.3172857)--	(-0.0971,0.3252449)--	(-0.0841,0.3332041)--	(-0.0711,0.34525)--	(-0.0581,0.3591786)--	(-0.0451,0.3731071)--	(-0.0321,0.3870357)--	(-0.0191,0.4009643)--	(-0.0061,0.4148929)--	(0.0069,0.4288214)--	(0.0199,0.44275)--	(0.0329,0.4616875)--	(0.0459,0.51925)--	(0.0589,0.5233)--	(0.0719,0.4843)--	(0.0849,0.448825)--	(0.0979,0.439075)--	(0.1109,0.429325)--	(0.1239,0.419575)--	(0.1369,0.409825)--	(0.1499,0.400075)--	(0.1629,0.390325)--	(0.1759,0.380575)--	(0.1889,0.370825)--	(0.2019,0.361075)--	(0.2149,0.351325)--	(0.2279,0.341575)--	(0.2409,0.331825)--	(0.2539,0.322075)--	(0.2669,0.312325)--	(0.2799,0.302575)--	(0.2929,0.2928571)--	(0.3059,0.2928571)--	(0.3189,0.2928571)--	(0.3319,0.2928571)--	(0.3449,0.2928571)--	(0.3579,0.2928571)--	(0.3709,0.2928571)--	(0.3839,0.2928571)--	(0.3969,0.2928571)--	(0.4099,0.2928571)--	(0.4229,0.2928571)--	(0.4359,0.2928571)--	(0.4489,0.2928571)--	(0.4619,0.2928571)--	(0.4749,0.2928571)--	(0.4879,0.2928571)--	(0.5009,0.2928571)--	(0.5139,0.2928571)--	(0.5269,0.2928571)--	(0.5399,0.2928571)--	(0.5529,0.2928571)--	(0.5659,0.2928571)--	(0.5789,0.2928571)--	(0.5919,0.2928571)--	(0.6049,0.2928571)--	(0.6179,0.2928571)--	(0.6309,0.2928571)--	(0.6439,0.2928571)--	(0.6569,0.2928571)--	(0.6699,0.2928571)--	(0.6829,0.2928571)--	(0.6959,0.2928571)--	(0.7089,0.2928571)--	(0.7219,0.2928571)--	(0.7349,0.2928571)--	(0.7479,0.2928571)--	(0.7609,0.2928571)--	(0.7739,0.2928571)--	(0.7869,0.2928571)--	(0.7999,0.2928571);

\draw[smooth,blue] (-0.5,0.3928571)--	(-0.1621,0.3928571)--	(-0.1491,0.3930224)--	(-0.1361,0.3954102)--	(-0.1231,0.397798)--	(-0.1101,0.4001857)--	(-0.0971,0.4025735)--	(-0.0841,0.4049612)--	(-0.0711,0.41525)--	(-0.0581,0.4291786)--	(-0.0451,0.4431071)--	(-0.0321,0.460051)--	(-0.0191,0.479949)--	(-0.0061,0.4998469)--	(0.0069,0.5197449)--	(0.0199,0.5396429)--	(0.0329,0.5616875)--	(0.0459,0.61925)--	(0.0589,0.63665)--	(0.0719,0.61715)--	(0.0849,0.6002937)--	(0.0979,0.6027312)--	(0.1109,0.6051687)--	(0.1239,0.6076062)--	(0.1369,0.6100437)--	(0.1499,0.6124813)--	(0.1629,0.5955687)--	(0.1759,0.5785062)--	(0.1889,0.5614437)--	(0.2019,0.5443812)--	(0.2149,0.5273188)--	(0.2279,0.5102562)--	(0.2409,0.4931938)--	(0.2539,0.4761313)--	(0.2669,0.4590687)--	(0.2799,0.4454321)--	(0.2929,0.4356959)--	(0.3059,0.4301245)--	(0.3189,0.4245531)--	(0.3319,0.4189816)--	(0.3449,0.4134102)--	(0.3579,0.4078388)--	(0.3709,0.4022673)--	(0.3839,0.3966959)--	(0.3969,0.3928571)--	(0.4099,0.3928571)--	(0.4229,0.3928571)--	(0.4359,0.3928571)--	(0.4489,0.3928571)--	(0.4619,0.3928571)--	(0.4749,0.3928571)--	(0.4879,0.3928571)--	(0.5009,0.3928571)--	(0.5139,0.3928571)--	(0.5269,0.3928571)--	(0.5399,0.3928571)--	(0.5529,0.3928571)--	(0.5659,0.3928571)--	(0.5789,0.3928571)--	(0.5919,0.3928571)--	(0.6049,0.3928571)--	(0.6179,0.3928571)--	(0.6309,0.3928571)--	(0.6439,0.3928571)--	(0.6569,0.3928571)--	(0.6699,0.3928571)--	(0.6829,0.3928571)--	(0.6959,0.3928571)--	(0.7089,0.3928571)--	(0.7219,0.3928571)--	(0.7349,0.3928571)--	(0.7479,0.3928571)--	(0.7609,0.3928571)--	(0.7739,0.3928571)--	(0.7869,0.3928571)--	(0.7999,0.3928571);

\draw[smooth,blue] (-0.5,0.65)--	(-0.1621,0.65)--	(-0.1491,0.65)--	(-0.1361,0.65)--	(-0.1231,0.65)--	(-0.1101,0.65)--	(-0.0971,0.65)--	(-0.0841,0.65)--	(-0.0711,0.65)--	(-0.0581,0.65)--	(-0.0451,0.65)--	(-0.0321,0.6530153)--	(-0.0191,0.6589847)--	(-0.0061,0.6649541)--	(0.0069,0.6709235)--	(0.0199,0.6768929)--	(0.0329,0.6870125)--	(0.0459,0.73925)--	(0.0589,0.75665)--	(0.0719,0.73715)--	(0.0849,0.7206462)--	(0.0979,0.7260087)--	(0.1109,0.7313712)--	(0.1239,0.7367337)--	(0.1369,0.7420963)--	(0.1499,0.7474588)--	(0.1629,0.7412112)--	(0.1759,0.7348737)--	(0.1889,0.7285362)--	(0.2019,0.7221988)--	(0.2149,0.7158613)--	(0.2279,0.7095238)--	(0.2409,0.7031863)--	(0.2539,0.6968488)--	(0.2669,0.6905112)--	(0.2799,0.6870123)--	(0.2929,0.6867452)--	(0.3059,0.6899289)--	(0.3189,0.6931125)--	(0.3319,0.6962962)--	(0.3449,0.6994799)--	(0.3579,0.7026636)--	(0.3709,0.7058472)--	(0.3839,0.7090309)--	(0.3969,0.7129571)--	(0.4099,0.7185286)--	(0.4229,0.7241)--	(0.4359,0.7296714)--	(0.4489,0.7352429)--	(0.4619,0.7306143)--	(0.4749,0.7250429)--	(0.4879,0.7194714)--	(0.5009,0.7139)--	(0.5139,0.7083286)--	(0.5269,0.7027571)--	(0.5399,0.6971857)--	(0.5529,0.6916143)--	(0.5659,0.6860429)--	(0.5789,0.6804714)--	(0.5919,0.6749)--	(0.6049,0.6693286)--	(0.6179,0.6637571)--	(0.6309,0.6581857)--	(0.6439,0.6526143)--	(0.6569,0.65)--	(0.6699,0.65)--	(0.6829,0.65)--	(0.6959,0.65)--	(0.7089,0.65)--	(0.7219,0.65)--	(0.7349,0.65)--	(0.7479,0.65)--	(0.7609,0.65)--	(0.7739,0.65)--	(0.7869,0.65)--	(0.7999,0.65);
     
     \draw[smooth,black,thick] (0.05,-0.5) -- (0.05,1);
     \draw[smooth,black,thick] (0.15,-0.5) -- (0.15,1);
     \draw[smooth,black,thick] (0.45,-0.5) -- (0.45,1);
     
     \draw[smooth,dotted,black,thick] (0.293,-0.5) -- (0.293,1);
     \draw[smooth,dotted,black,thick] (0.393,-0.5) -- (0.393,1);
     \draw[smooth,dotted,black,thick] (0.65,-0.5) -- (0.65,1);
     
     \draw[smooth,dashed,red,thick] (-0.15,-0.5) -- (-0.15,1);
     \draw[smooth, dashed,red,thick] (-0.08,-0.5) -- (-0.08,1);
     \draw[smooth, dashed,red,thick] (-0.0387,-0.5) -- (-0.0387,1);
	 \draw[smooth, dashed,red,thick] (0.0329,-0.5) -- (0.0329,1);     
     \draw[smooth, dashed,red,thick] (0.04,-0.5) -- (0.04,1);     
     
     \draw[smooth, dashdotted,brown,thick] (0.083,-0.5) -- (0.083,1);
     \draw[smooth, dashdotted,brown,thick] (0.2738,-0.5) -- (0.2738,1);
     
     \draw[] (0.8,0.8) node[right,green] {$a_1$};
     \draw[] (0.8,1) node[right,green] {$d_1$};
     \draw[] (0.8,0.05) node[right,blue] {$a_2$};
     \draw[] (0.8,0.283) node[right,blue] {$d_2$};
     \draw[] (0.8,0.15) node[right,blue] {$a_3$};
     \draw[] (0.8,0.393) node[right,blue] {$d_3$};
     \draw[] (0.8,0.48) node[right,blue] {$a_4$};
     \draw[] (0.8,0.65) node[right,blue] {$d_4$};
\end{tikzpicture}
\caption{}\label{fig:departure_arrival_a}
\end{subfigure}

 \bigskip

\begin{subfigure}{\columnwidth}
\begin{tikzpicture}[xscale=10,yscale=2.7]
  \def\xmin{-0.5}
  \def\xmax{0.8}
  \def\ymin{1}
  \def\ymax{2.1}
    \draw[->] (\xmin,\ymin) -- (\xmax,\ymin) node[right] {$x=a_1$} ;
    \draw[->] (\xmin,\ymin) -- (\xmin,\ymax); 
    \foreach \x in {-0.5,0,0.5}
    \node at (\x,\ymin) [below] {\x};
    \foreach \y in {1,2}
    \node at (\xmin,\y) [left] {\y};
    
    \draw[smooth,blue, thick] (-0.5,1.602602)--	(-0.4871,1.582128)--	(-0.4741,1.561833)--	(-0.4611,1.541875)--	(-0.4481,1.522256)--	(-0.4351,1.502974)--	(-0.4221,1.48403)--	(-0.4091,1.465425)--	(-0.3961,1.447157)--	(-0.3831,1.429228)--	(-0.3701,1.411636)--	(-0.3571,1.394382)--	(-0.3441,1.377467)--	(-0.3311,1.360889)--	(-0.3181,1.34465)--	(-0.3051,1.328748)--	(-0.2921,1.313184)--	(-0.2791,1.297959)--	(-0.2661,1.283071)--	(-0.2531,1.268522)--	(-0.2401,1.25431)--	(-0.2271,1.240436)--	(-0.2141,1.226901)--	(-0.2011,1.213703)--	(-0.1881,1.200844)--	(-0.1751,1.188322)--	(-0.1621,1.176138)--	(-0.1491,1.164785)--	(-0.1361,1.160652)--	(-0.1231,1.157346)--	(-0.1101,1.154868)--	(-0.0971,1.153218)--	(-0.0841,1.152396)--	(-0.0711,1.1649)--	(-0.0581,1.185726)--	(-0.0451,1.20944)--	(-0.0321,1.242738)--	(-0.0191,1.285839)--	(-0.0061,1.332305)--	(0.0069,1.382133)--	(0.0199,1.435326)--	(0.0329,1.504574)--	(0.0459,1.762763)--	(0.0589,1.84041)--	(0.0719,1.727486)--	(0.0849,1.63266)--	(0.0979,1.638455)--	(0.1109,1.644987)--	(0.1239,1.652255)--	(0.1369,1.660259)--	(0.1499,1.668999)--	(0.1629,1.613784)--	(0.1759,1.558953)--	(0.1889,1.505008)--	(0.2019,1.451948)--	(0.2149,1.399775)--	(0.2279,1.348487)--	(0.2409,1.298084)--	(0.2539,1.248568)--	(0.2669,1.199937)--	(0.2799,1.160806)--	(0.2929,1.132053)--	(0.3059,1.130875)--	(0.3189,1.130005)--	(0.3319,1.129443)--	(0.3449,1.129188)--	(0.3579,1.129241)--	(0.3709,1.129601)--	(0.3839,1.130269)--	(0.3969,1.136921)--	(0.4099,1.156751)--	(0.4229,1.177332)--	(0.4359,1.198665)--	(0.4489,1.22075)--	(0.4619,1.213327)--	(0.4749,1.203303)--	(0.4879,1.193453)--	(0.5009,1.183774)--	(0.5139,1.174268)--	(0.5269,1.164935)--	(0.5399,1.155774)--	(0.5529,1.146785)--	(0.5659,1.137969)--	(0.5789,1.129326)--	(0.5919,1.120854)--	(0.6049,1.112556)--	(0.6179,1.104429)--	(0.6309,1.096475)--	(0.6439,1.088694)--	(0.6569,1.08998)--	(0.6699,1.099428)--	(0.6829,1.109214)--	(0.6959,1.119339)--	(0.7089,1.129801)--	(0.7219,1.140602)--	(0.7349,1.15174)--	(0.7479,1.163216)--	(0.7609,1.175031)--	(0.7739,1.187183)--	(0.7869,1.199674)--	(0.7999,1.212502);

	 \draw[smooth,black,thick] (0.05,\ymin) -- (0.05,\ymax);
     \draw[smooth,black,thick] (0.15,\ymin) -- (0.15,\ymax);
     \draw[smooth,black,thick] (0.45,\ymin) -- (0.45,\ymax);
     
     \draw[smooth,dotted,black,thick] (0.293,\ymin) -- (0.293,\ymax);
     \draw[smooth,dotted,black,thick] (0.393,\ymin) -- (0.393,\ymax);
     \draw[smooth,dotted,black,thick] (0.65,\ymin) -- (0.65,\ymax);
     
     \draw[smooth,dashed,red,thick] (-0.15,\ymin) -- (-0.15,\ymax);
     \draw[smooth, dashed,red,thick] (-0.08,\ymin) -- (-0.08,\ymax);
     \draw[smooth, dashed,red,thick] (-0.0387,\ymin) -- (-0.0387,\ymax);
	 \draw[smooth, dashed,red,thick] (0.0329,\ymin) -- (0.0329,\ymax);     
     \draw[smooth, dashed,red,thick] (0.04,\ymin) -- (0.04,\ymax);     
     
     \draw[smooth, dashdotted,brown,thick] (0.083,\ymin) -- (0.083,\ymax);
     \draw[smooth, dashdotted,brown,thick] (0.2738,\ymin) -- (0.2738,\ymax);        

  	\draw[] (0.7,1.4) node[right,blue] {$\sum_{i\in\mathcal{N}}c_i(\mathbf{a})$};	

\end{tikzpicture}
\caption{}\label{fig:departure_arrival_xx}
\end{subfigure}

\caption{
{\bf (a)} Arrival (horizontal dotted) and departure (horizontal solid) profiles obtained by changing the arrival time of user $1$. {\bf (b)} Cost function obtained by changing the arrival time of user $1$. Break points are marked in both (a) and (b) by vertical lines as follows: Solid black lines mark Type 1a points (note there are exactly $N-1=3$ such breakpoints). Dotted black lines mark Type 1b breakpoints (note that there are exactly $N-1=3$ such breakpoints as well). Type 2a breakpoints are marked by dashed red lines and Type 2b breakpoints are marked by brown dashed-dotted lines.
\label{fig:departure_arrival_xx}}
\end{figure}

As is further evident from Figure~\ref{fig:departure_arrival_xx}~(a), the dynamics of the departure times are piecewise affine with breakpoints as marked by the vertical lines in the figure. In between these lines, the effect of changing $x$ on other quantities is affine. In between these breakpoints, the objective function is piecewise convex (quadratic). This property is illustrated in Figure~\ref{fig:departure_arrival_xx}~(b) where the objective is plotted as a function of $x$. This property allows us to optimise globally over a single coordinate, utilizing the problem structure. The desired departure times used for the cost function in (b) were $d^*_i=0.5$ for $i=1,\ldots,4$.

The global search over a single coordinate works by varying $x$ from $\underline{a}$ to $\overline{a}$ and in the process searches for the one-coordinate optimum. This is done with a finite number of steps because of the piecewise-affine dynamics. Our algorithm incrementally computes the piecewise-affine coefficients within these steps. We call each step a ``breakpoint''. The following types of breakpoints may occur:

\begin{description}
\item {\bf Type 1a:} The arrival of $r$ overtakes the next arrival of any $i$\\
(solid black line).
\item {\bf Type 1b:} The departure of any $i$ is overtaken by the arrival of $r$\\
(dotted black line).
\item {\bf Type 2a:} The departure of any $i$ overtakes any arrival\\
(dashed red line).
\item {\bf Type 2b:} The departure of any $i$ is overtaken by an arrival of $j\neq r$\\
(brown dashed-dotted line).
\end{description}

Observe that in varying $x$, breakpoints of type 1a and 1b occur exactly $N-1$ times each. Less trivially, we have a bound on the number of type 2a and 2b breakpoints:

\begin{proposition}\label{prop:cpi_breakpoints}
In executing the global search over a single coordinate $r$, the total number of breakpoints is  $O(N^3)$. 
\end{proposition} 
Before presenting the proof, we present the details of the piecewise-affine dynamics and the details of the global search over a single coordinate $r$ algorithm.

\subsection{Algorithm Details}
\label{sec:algDetails1}

In carrying out the global search over a single coordinate $r$, we remove the restriction that arrival times are ordered. That is, the search region is extended from ${\cal R}$ to a set not requiring such order $\widetilde{\cal R} := [\underline{a},\,\overline{a}]^N$.
This allows us to carry out a full search for the optimum with respect to a single user $r$ without the restriction $a_{r} \in [a_{r-1},\, a_{r+1}]$. This broader search potentially enables bigger gains in the objective when integrating the algorithm within a search heuristic. Further, any point ${\mathbf a} \in \widetilde{\cal R}$ can be  mapped into a unique point ${\cal O}({\mathbf a})  \in {\cal R}$ where ${\cal O}(\cdot)$ is an ordering operator. By Lemma~\ref{Lemma:optimal_order} we have that $ c\big( {\cal O}({\mathbf a}) \big) \le c\big( {\mathbf a} \big)$.

Take $\widetilde{\mathbf a} \in \widetilde{\cal R}$ as an initial arrival vector and suppose that we are optimising over user $r$. Let $x \in  [\underline{a},\,\overline{a}]$ be the immediate search value of $a_r$ (keeping the other arrival times fixed). For any such $x$ we define a corresponding permutation ${\boldsymbol \pi}(\widetilde{\mathbf a}, \, x)$ indicating the current order of arrivals, as well as the ordered arrival vector
\begin{equation}
\label{eq:325}
{\mathbf a}(\widetilde{a}, \, x) := {\cal O} \big(  
a_{\pi_1(x)},\, \ldots, a_{\pi_{r}(x)-1},
x,\,
a_{\pi_{r}(x)+1},\, \ldots, a_{\pi_N(x)}
\big).
\end{equation}
This vector can serve as input to Algorithm~1a yielding a corresponding ${\mathbf d}(\widetilde{a}, \, x)$, ${\mathbf k}(\widetilde{a}, \, x)$ and ${\mathbf h}(\widetilde{a}, \, x)$. Furthermore, using \eqref{eq:dRecursion} we have the local piecewise-affine relationship,
\[
{ d}_i(\widetilde{a}, \, x) =x \, \theta_{i\, | r, {\boldsymbol \pi}(\widetilde{\mathbf a}, \,x), {\mathbf k}(\widetilde{\mathbf a}, \,x)} +\eta_{i\, | r, {\boldsymbol \pi}(\widetilde{\mathbf a}, \,x), {\mathbf k}(\widetilde{\mathbf a}, \,x)},\ i\in\mathcal{N}, \qquad x \in [\underline{a},\, \overline{a}].
\]
That is, the coefficients of the departures between breakpoints depend on the permutation of the users as well as on the current order of their arrivals and departures. For brevity we omit the dependencies on $x$, $\widetilde{\mathbf a}$, ${\boldsymbol \pi}$ and ${\mathbf k}$. Manipulating \eqref{eq:dRecursion} we obtain, 
\begin{equation}\label{eq:theta_single}
\big(\theta_i, \, \eta_i \big)=
\left\{
	\begin{array}{ll}
		\Big(
		0,\,
		\frac{1+a_i(\beta-\alpha(i-h_i))-\alpha\left(\sum_{j=i+1}^{k_i}a_j-\sum_{j=h_i}^{i-1}\eta_j\right)}{\beta-(k_i-i)\alpha}
		\Big), &  i<\pi_r \ , \ k_i< \pi_r, \\
		\Big(
		-\frac{\alpha\left(1-\sum_{j=h_i}^{i-1}\theta_j\right)}{\beta-\alpha(k_i-i)},\,
		\frac{1+a_i(\beta-\alpha(i-h_i))-\alpha\left(\sum_{j=i+1}^{k_i}a_j\mathbbm{1}\{j\neq \pi_r\}-\sum_{j=h_i}^{i-1}\eta_j\right)}{\beta-\alpha(k_i-i)}
		\Big),
		 & i<\pi_r \ , \ k_i\geq \pi_r, \\
\Big(
		\frac{\beta+\alpha\sum_{j=h_i}^{i-1}(\theta_j-1)}{\beta-\alpha(k_i-i)} ,\,
		\frac{1-\alpha\left(\sum_{j=i+1}^{k_i}a_j-\sum_{j=h_i}^{i-1}\eta_j\right)}{\beta-\alpha(k_i-i)}
		\Big),
		&  i=\pi_r, \\
		\Big(
		\frac{\alpha\sum_{j=h_i}^{i-1}\theta_j}{\beta-\alpha(k_i-i)}, \,
		\frac{1+a_i(\beta-\alpha(i-h_i))-\alpha\left(\sum_{j=i+1}^{k_i}a_j-\sum_{j=h_i}^{i-1}\eta_j\right)}{\beta-\alpha(k_i-i)}
		\Big), &  i>\pi_r. \\
	\end{array}
\right.
\end{equation}

On every interval, the departure times $d_i$ are all affine and continuous w.r.t $x$ with the above coefficients, until a breakpoint (of type 1a, 1b, 2a or 2b) occurs. Computing the time of the next breakpoint is easily done by considering the piecewise affine dynamics. Potential breakpoints of types 1a and 1b are to occur at times $t$ where $x+t=a_{\pi_r+1}$ and $t \, \theta_i+d_i=a_{r}+t$, respectively. Potential breakpoints of types 2a and 2b involving user $i$ are to occur at times $t \, \theta_i+d_i=a_{k_i+1}$ and $t \, \theta_i+d_i=a_{k_i}$ respectively. Observing now that type 2a breakpoints may occur only when $\theta_i>0$ and type 2b breakpoints may occur only when $\theta_i<0$ we have that the next breakpoint occurs at, 
\begin{equation}
\label{eq:343}
\tau = \min\{ t_0,t_1,\dots,t_N,t_{N+1}\},
\end{equation}
where $t_0=a_{\pi_r+1}-x$ (type 1a breakpoints), $t_{N+1}=\bar{a}-x$ (termination) and for $1\leq i\leq N$:
\[
t_i=
\left\{
	\begin{array}{ll}
		\frac{a_{k_i}-\theta_i x-\eta_i}{\theta_i} & \mbox{, } 		\theta_i<0, \ k_i\neq r , \\
				\frac{a_{k_i}-\theta_i x-\eta_i}{\theta_i-1} & \mbox{, } 		\theta_i<0, \ k_i=r,  \\
		\frac{a_{k_i+1}-\theta_i x-\eta_i}{\theta_i} & \mbox{, } \theta_i>0 , 	\\
		\infty & \mbox{, } \theta_i=0. 	\\
	\end{array}	
\right. 
\]
Considering the time interval until the next breakpoints, $[x,\tau]$ we have that the total cost as a function of the arrival time $\hat{x} \in [x,\tau]$ of user $r$ is
\begin{equation*}\label{Algo_r_cost}
{\tilde c}(\hat{x};{\boldsymbol \pi}):=\sum_{j\in\mathcal{N}}\left((\theta_{\pi_j}\hat{x}+\eta_{\pi_j}-d_{j}^*)^2+\gamma(\theta_{\pi_j}\hat{x}+\eta_{\pi_j}-a_{\pi_j})\right),
\end{equation*}
with derivative $\partial {\tilde c}(\hat{x};{\boldsymbol \pi})=\sum_{j\in\mathcal{N}}\theta_{\pi_j}\left(2(\eta_{\pi_j}-d_{j}^*)+\gamma\right)+2\hat{x}\sum_{j\in\mathcal{N}}\theta_{\pi_j}^2$, and with the root $x_0 \ge x$, solving $\partial {\tilde c}(x_0;{\boldsymbol \pi})=0$  (and often not lying within the interval $[x,\tau]$):
\begin{equation*}\label{Algo_root}
x_0=\frac{-\sum_{j\in\mathcal{N}}\theta_{\pi_j}\left(2(\eta_{\pi_j}-d_{j}^*)+\gamma\right)}{2\sum_{j\in\mathcal{N}}\theta_{\pi_j}^2}.
\end{equation*}
Note that it is crucial to keep track of ${\boldsymbol \pi}$ at every step in order to associate the correct ideal departure time to every user. In iterating over intervals we search for the minimal ${\tilde c}(\hat{x};{\boldsymbol \pi})$ (denoted $m^*$) as follows: If $\partial {\tilde c}(\hat{x};{\boldsymbol \pi})>0$ for all $\hat{x} \in [x,\tau]$, then we continue to the next interval. Otherwise, if $x_0-x\leq \tau$ and $m^*>{\tilde c}(x_0;{\boldsymbol \pi})$, then set $m^*={\tilde c}(x_0;{\boldsymbol \pi})$ and $\big(\mathbf{a}^*\big)_r=x_0$, and if $x_0-x>\tau$ and $m^*>{\tilde c}(x+\tau;{\boldsymbol \pi})$, then set $m^*={\tilde c}(x+\tau;{\boldsymbol \pi})$ and $\big(\mathbf{a}^*\big)_r=x+\tau$. 

In this way, $x$ updates over intervals, of the form $[x,\tau]$. Prior to moving to the next interval we need to update the permutation variables $\mathbf{\pi}$, $\mathbf{k}$, and $\mathbf{h}$. Denote the minimizing set of \eqref{eq:343} by $\mathcal{T}:=\argmin\lbrace t_0,t_1,\dots,t_N\rbrace$ and sequentially for every $i\in\mathcal{T}$:
\begin{itemize}
\item If $i=0$, then we update $\pi$ by changing the order between user $r$ and the next user $j:\pi_j=\pi_r+1$, i.e. set $\pi_r=\pi_r+1$ and $\pi_j=\pi_j-1$. In this case, there is no change in $\mathbf{k}$ or $\mathbf{h}$. (Type 1a breakpoints).
\item If $i\in\{1,\ldots,N\}$, then the order ${\boldsymbol \pi}$ does not change, but we update $\mathbf{k}$ and $\mathbf{h}$: If $\theta_{i}<0$, then update $h_{k_{i}}=h_{k_{i}}+1$, followed by $k_{i}=k_{i}-1$. If $\theta_{i}>0$, then update $k_{i}=k_{i}+1$, followed by $h_{k_{i}	}=h_{k_{i}}-1$. (All other types of breakpoints).
\item If $i=N+1$, then the iteration is complete and no changes are required.
\end{itemize} 

\begin{remark}
For any convex and differentiable cost functions, the first order condition yielding $x_0$ can be solved. For some elaborate functions this may also require a numerical procedure. If the late and early cost functions are not strictly convex (for example affine), then computing $x_0$ can be skipped. If the cost function is piecewise affine, then only the sign of $\partial {\tilde c}$ needs to be computed, and if it is negative check if the next point $x+\tau$ is a new minimum point or not.
\end{remark}

\begin{algorithm}[H]
\renewcommand{\thealgorithm}{}
\caption{\textbf{4}: Global search over a single coordinate}
\label{algo:Alg4}
\textbf{Input}: $\tilde{\mathbf a} \in {\widetilde{\cal R}}$, $r \in {\cal N}$, ${\boldsymbol \pi}$\\ 
\textbf{Output}: $\mathbf{a}^*$ and $m^*$
\begin{algorithmic}
\State init $x=\tilde{a}_r=\underline{a}$
\State init $a={\cal O}(\tilde{\mathbf{a}})$
\State run \textbf{Alg.1a}$(a)$ $\rightarrow\left(\mathbf{d}, \mathbf{k},\mathbf{h}\right)$
\State init $\mathbf{a}^*= \tilde{\mathbf{a}}$
\State init $m^*={\tilde c}(x;{\boldsymbol \pi})$
\State set $\pi_r=1$
\For{$i<r$} \State set $\pi_i=\pi_i+1$ \EndFor	
\While{$x\leq\overline{a}$} 
	\State set $a={\cal O}(\tilde{\mathbf{a}})$
	\State compute: ${\boldsymbol\theta}$, ${\boldsymbol\eta}$, $\tau$, $\mathcal{T}$, and $\partial {\tilde c}(x;{\boldsymbol \pi})$
	\If{$\partial {\tilde c}(x;{\boldsymbol \pi})<0$}
		\State compute $x_0$ and ${\tilde c}(x_0;{\boldsymbol \pi})$
		\If{$x_0<x+\tau$}
			\If{${\tilde c}(x_0;{\boldsymbol \pi})<m^*$}
				\State set $a^*_r=x_0$
				\State set $m^*={\tilde c}(x_0;{\boldsymbol \pi})$			
			\EndIf
		\ElsIf{${\tilde c}(x+\tau;{\boldsymbol \pi})<m^*$}
			\State set $a^*_r=x+\tau$
			\State set $m^*={\tilde c}(x+\tau;{\boldsymbol \pi})$
		\EndIf
	\EndIf
	\State set $x=x+\tau$	
	\For{$i\in \mathcal{T}$}
		\If{$i=0$}				
			\State set $\pi_j=\pi_j-1$ where $j$ satisfies $\pi_j=\pi_r+1$ 
			\State set $\pi_r=\pi_r+1$
		\EndIf	
		\If{$i\in\{1,\ldots,N\}$}	
			\If{$\theta_{i}<0$}
				\State set $h_{k_{i}}=h_{k_{i}}+1$ and $k_{i}=k_{i}-1$
			\ElsIf{$\theta_{i}>0$}
				\State set $k_{i}=k_{i}+1$ and $h_{k_{i}}=h_{k_{i}}-1$
			\EndIf
		\EndIf				
	\EndFor	
\EndWhile	
\State \textbf{return} ({$\mathbf{a}^*$,\, $m^*$})
\end{algorithmic}
\end{algorithm}

\subsection{Computational Complexity}
\label{sec:runningTimes}

In the following series of lemmata we analyse the complexity of Algorithm 4. In particular, we prove Proposition \ref{prop:cpi_breakpoints}, establishing bounds for the number of breakpoints of each type. Throughout the analysis we continue denoting the coordinate being optimised by $r$ and the respective value by $x=\tilde{a}_r$. Keep in mind that $\pi_i$, $\theta_i$, $k_i$, and $h_i$ are functions of $x$ and the initial unordered vector $\tilde{\mathbf{a}}$ for every $i\in\mathcal{N}$. We treat ${\mathbf a}$ as the ordered vector \eqref{eq:325} as before.

\begin{lemma}\label{Lemma_coeff_decreasing}
For any $i\in\mathcal{N}$ such that $i\neq r$, the coefficient $\theta_i \le 0$ and as a consequence $d_i(x)$ is monotone non-increasing for every $x>a_i$.
\end{lemma}
\begin{lemma}\label{Lemma_coeff_changes}
For any permutation ${\boldsymbol \pi}$ at the start of the global search on $r$, the coefficient $\theta_i$ of any $i\in{\boldsymbol \pi}$ changes sign from strictly positive to strictly negative or vice versa at most $i-1$ times during the search. 
\end{lemma}
We now prove proposition~\ref{prop:cpi_breakpoints}:
\begin{proof}
For any $2\leq i\leq N$ in the original permutation ${\boldsymbol \pi}$, the type ~2a and ~2b breakpoints occur at most $N-i$ times for every change of sign. This is because their departure time can only cross arrival times of later arrivals. According to Lemma \ref{Lemma_coeff_changes}, the number of sign changes for any $2\leq i\leq N$ is at most $i-1$. Thus, the total number of breakpoints of type (2a or ~2b) is at most
\[
\sum_{i=2}^N (N-i)(i-1)=\frac{N(2-3N+N^2)}{6}.
\]
Thus, adding up all types of breakpoints, we get that the search domain $[\underline{a},\overline{a}]$ is broken up to at most $\left(\frac{1}{3}N^3-N^2+\frac{8}{3}N-2\right)$ intervals.
\qed
\end{proof}

Furthermore, we have the following bound for the complexity of Algorithm~4.

\begin{corollary}\label{prop:alg4_complexity}
The computation complexity of Algorithm 4 is at most $O(N^5)$.
\end{corollary}
\begin{proof}
In every interval step of a global search on a single coordinate there is a need to compute the coefficient vectors ${\boldsymbol \eta}$ and ${\boldsymbol \theta}$. This is equivalent to calculating the departure times recursively using Algorithm 1a. In Proposition \ref{prop:alg1_complexity} it was shown that the recursion requires at most $2N$ steps. On top of this, in every one of these steps the actual computation requires summation of up to $N$ variables. Now since the number of breakpoints intervals is bounded by $O(N^3)$ we conclude the result.
\qed
\end{proof}

\subsection{Coordinate Pivot Iteration Optimization}
\label{sec:CPIAlg}

In this subsection we illustrate how Algorithm 4 can be applied to carry out standard Coordinate Pivot Iteration (CPI), see \cite{Bertsekas1999}, p272. In every iteration of the CPI algorithm, the total cost function is minimized with respect to the arrival time of one user, when all other arrival times are fixed. This is then repeated for all users; we call the iteration over all $N$ users a CPI cycle. 
The CPI algorithm stops when the total improvement in a cycle is smaller than some specified tolerance parameter, $\epsilon >0$. Note that in non-smooth CPI (such as our case), CPI often stops when the total improvement is in-fact exactly $0$. That is, $\epsilon$ is often not a significant parameter. A further comment is that our CPI algorithm utilizes Algorithm~4 searching over the broader space, $\widetilde{\cal R}$. We can thus improve the objective (see Lemma \ref{Lemma:optimal_order}) by
incorporating the ordering operator, ${\cal O}$, at the end of each CPI cycle. 

We add the following notations for the optimization procedure: Let $n=0,1,\dots$ be the cycle number, $c^{(n)}$ the total cost at end of cycle $n$, $m^*$ the global minimal total cost, and $\mathbf{a}^*$ the global optimal arrival vector.

\begin{algorithm}[H]
\renewcommand{\thealgorithm}{}
\caption{\textbf{5}: Coordinate pivot iteration (global search)}
\textbf{Input}: $\mathbf{a}^{(0)}$ and $\epsilon$ \\
\textbf{Output}: $\mathbf{a}^*$ and $m^*$
\begin{algorithmic}
\State init $n=0$
\State init $\Delta=\epsilon+1$
\State init $\mathbf{a}^*=\mathbf{a}^{(0)}$
\State init $c^{(0)} = c(\mathbf{a}^*)$
\While{$\Delta>\epsilon$}
	\State set $n=n+1$
	\State set $\tilde{\mathbf a} = {\mathbf a}^*$
	\For{$r\in\mathcal{N}$}
		\State run \textbf{Alg. 4}$(r, \tilde{\mathbf{a}})$ $\rightarrow \tilde{\mathbf{a}},$
	\EndFor
	\State set ${\mathbf a}^* = {\cal O}\big(\tilde{\mathbf a}\big)$
	\State set ${c}^{(n)}=c({\mathbf a}^*)$
	\State set $\Delta={c}^{(n-1)}-{c}^{(n)}$
\EndWhile
\State set $m^*={c}^{(n)}$
\State \textbf{return} {$(\mathbf{a}^*,\ m^*)$}
\end{algorithmic}
\label{algo:optim1}
\end{algorithm}


Hinging upon the results of the previous section, we have:

\begin{corollary}
The computation complexity of a single \rm{CPI} cycle, i.e. conducting a line search on all coordinates, is at most $O(N^6)$.
\end{corollary}
\begin{proof}
In Proposition \ref{prop:alg4_complexity} we established that for a single coordinate the complexity is at most $O(N^5)$. It is therefore immediate that the complexity of running the algorithm for every coordinate is at most $O(N^6)$. \qed
\end{proof}

Note that while we have a polynomial time CPI algorithm, there is no guarantee that it converges to a local minimum since the objective function is not smooth. In fact, numerical experimentation suggests that this is typically the case when the number of users is not very small, i.e., $N\geq 4$. Nevertheless, experimentation has shown that CPI algorithm generally outputs an arrival vector that lies in the vicinity of the optimum. This motivates combining it with the neighbour search, Algorithm~3 as discussed in the next section.

\section{A Combined Heuristic and Numerical Results}
\label{sec:combHeurist}
We now utilise the problem structure and aforementioned algorithms to produce a combined heuristic. We use ${\cal A}$ as in~\eqref{eq:3488} for initial points. For each of these $M$ initial points we run a CPI (global) search followed by neighbour (local) search. The core principal is to use the CPI method in order to find a ``good'' initial polytope, or equivalently an arrival-departure permutation, and then to seek a local minimum using the neighbour search.

\begin{algorithm}[H]
\renewcommand{\thealgorithm}{}
\caption{\textbf{6}: Combined global and local search heuristic}
\textbf{Input}: Model parameters only 
($N,\alpha,\beta,\mathbf{d}^*$ and $\gamma$) \\
\textbf{Output}: $\mathbf{a}^*$ (local optimum)
\begin{algorithmic}
\State init $m^*=\infty$
\For{$\mathbf{a}\in\mathcal{A}$}			
	\State run $\mathbf{Alg.5}(\mathbf{a},\ldots)\rightarrow(\hat{\mathbf{a}},\hat{m})$
	\State set $\hat{\mathbf{k}}=\mathbf{k}(\hat{\mathbf{a}})$
	\State run $\mathbf{Alg.3}(\tilde{\mathbf{k}},\ldots)\rightarrow(\hat{\mathbf{a}},\hat{m})$
	\If{$\hat{m}<m^*$}
		\State set $\mathbf{a}^*=\hat{\mathbf{a}}$ and $m^*=\hat{m}$
	\EndIf
\EndFor
\State \textbf{return} {$(\mathbf{a}^*,\ m^*)$}
\end{algorithmic}
\end{algorithm}

We tested the combined heuristic Algorithm~6 on a variety of problem instances and it appears to perform very well both in terms of running time and in finding what we believe is a global optimum. Here we illustrate these results for one such problem instance. We take $\beta = 1$ and $\alpha = 0.8/N$ (in this case the maximal slowdown is of the order of $80\%$ independently of $N$). We set ${\mathbf d}^*$ as the $N$ quantiles of a normal distribution with mean $0$ and standard deviation $1/2$. That is, there is an ideal departure profile centred around $0$. It is expected that when using optimal schedules, more congestion will occur as $N$ increases and/or $\gamma$ decreases.

Figure \ref{fig:opt_arrival_departure2} illustrates the dynamics of the obtained schedules as generated by the heuristic (using $M=3$ and $\epsilon=0.001$). In these plots arrival times of individual users are plotted on the top axis, marked by blue dots, shifted to the right by the free flow time ($1/\beta=1$). Departure times are plotted on the bottom axis. Users that do not experience any delay are then represented by lines that are exactly vertical. Further, the more slanted the line, the more slowdown that the user experiences. The ideal departure times are marked by green stars. Hence ideally the stars are to align with the red dots. This occurs exactly when $\gamma=0$, and approximately occurs for small~$\gamma$, for instance $\gamma=0.1$ as in (a) and (d). Then as $\gamma$ is increased, the optimal schedule is such that there is hardly any delay (almost perfectly vertical lines), but in this case, users experience major deviations between departure times and the ideal values.

For $N=15$, as presented in (a)--(c), we were indeed able to verify optimality using the exhaustive search Algorithm~2. For $N=50$, as presented in (d)--(f) we are not able to use the exhaustive search algorithm in any reasonable time. Nevertheless, in this case, in addition to seeing qualitatively sensible results, experimentation showed that increasing $M$ does not modify the results. Hence we believe that the obtained schedules are also optimal. 

For $N \le 15$, we were not able to find a case where the heuristic did not find the optimal schedule. {This was tested on a wide range of parameter values by varying $\alpha$ and $\gamma$ and randomly generating multiple due date vectors.} Further for large $N$ (up to $500$) we see insensitivity with respect to $M$ (the number of initial points) as well as to other randomized initial points. This result was also robust to changes in all of the parameter values ($\alpha$, $\beta$, $\gamma$, and $\mathbf{d}^*$). This leads us to believe that our heuristic performs very well.

Results, comparing running times are reported in Table~\ref{tbl:iter_compare} where we consider the algorithm with a single initial point ${\mathbf a}^0$ ($M=1$), and compare it to the exhaustive search given by Algorithm 2. For this table, we use the same problem data as described above, but scale the standard deviation by $N$ to be $0.04N$. For $N\leq 15$ the combined heuristic converged to the global optimum as verified by the exhaustive search with a negligible number of computations. For example, for $N=15$ the heuristic method made $\sim 737$ core computations, i.e. solving a QP for a single CPI interval or NS polytope, in 3.28 seconds, while the exhaustive search had to solve $\sim 10^7$ quadratic programs and required about $11$ hours \footnote{These computation times are using an AMD computer with 4 \textit{Phenom II 955 3.2GHz} processors, with our algorithms implemented in version 3.1.2 of the R software.}. Clearly, for larger $N$ it is not feasible to run the exhaustive search while the combined heuristic is still very quick, as seen for up to $N=50$ in Table~\ref{tblLife}. 

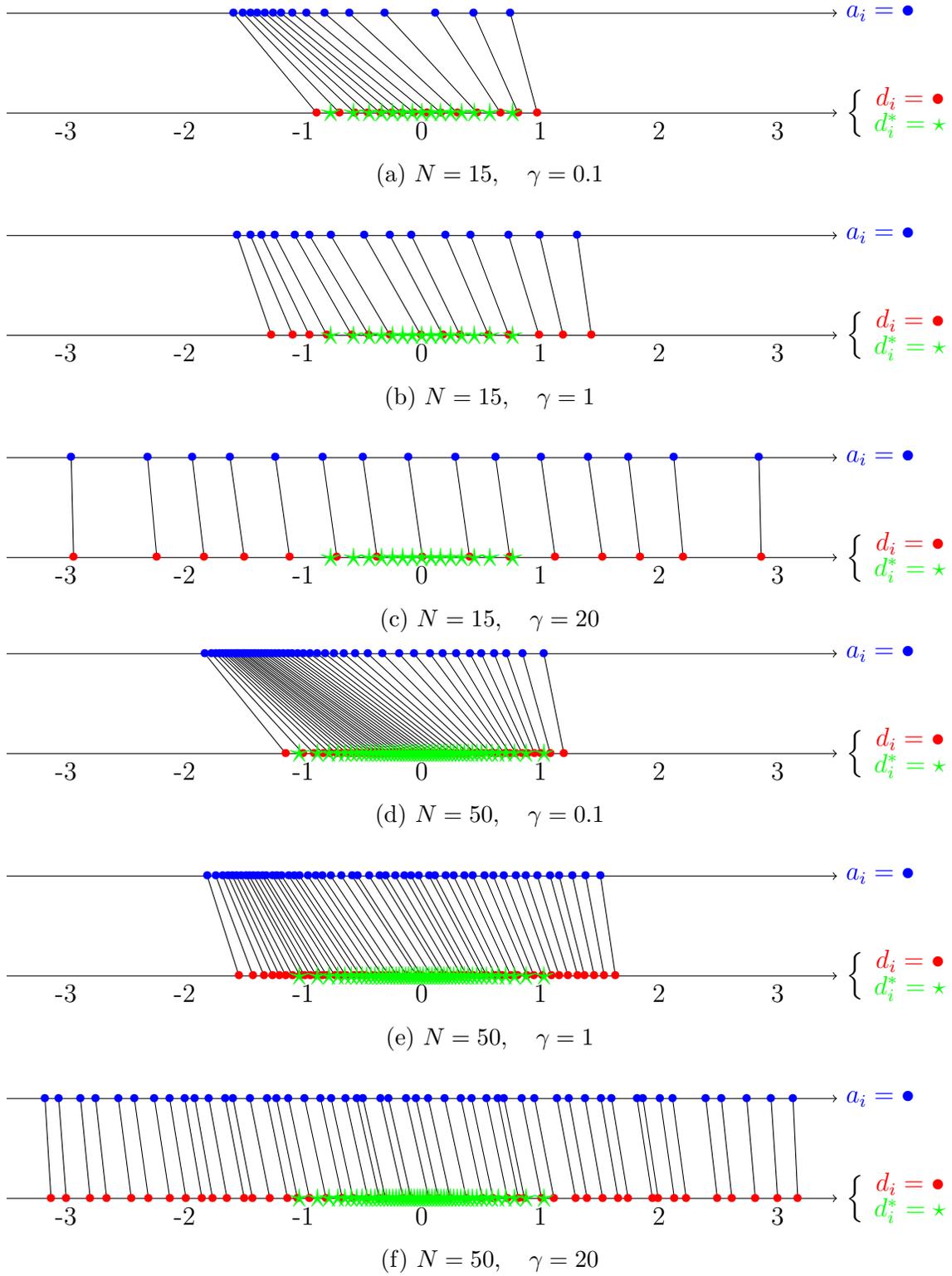
\begin{figure}[H]
\begin{subfigure}{\columnwidth}
\begin{tikzpicture}[xscale=1.9,yscale=1.6]
  \def\xmin{-3.5}
  \def\xmax{3.5}
  \def\ymin{0}
  \def\ymax{1}
    \draw[->] (\xmin,\ymin) -- (\xmax,\ymin) node[right] 
    {$\left\{\def\arraystretch{0.8}\begin{array}{c} \textcolor{red}{d_i=\bullet} \\ \textcolor{green}{d_i^*=\star} \end{array}\right.$} ;
    \draw[->] (\xmin,\ymax) -- (\xmax,\ymax) node[right] {$\textcolor{blue}{a_i=\bullet}$} ;
     \foreach \x in {-3,-2,-1,0,1,2,3}
    	\node at (\x,\ymin) [below] {\x};
    	
\draw[smooth,black] (-1.5875571,\ymax)--(-0.88602495,\ymin);
\draw[smooth,black] (-1.5073526,\ymax)--(-0.69198188,\ymin);
\draw[smooth,black] (-1.4438297,\ymax)--(-0.56310801,\ymin);
\draw[smooth,black] (-1.3869439,\ymax)--(-0.4542124,\ymin);
\draw[smooth,black] (-1.3214492,\ymax)--(-0.34841317,\ymin);
\draw[smooth,black] (-1.2524939,\ymax)--(-0.25380006,\ymin);
\draw[smooth,black] (-1.1844134,\ymax)--(-0.16672851,\ymin);
\draw[smooth,black] (-1.0902795,\ymax)--(-0.0645832,\ymin);
\draw[smooth,black] (-0.9723277,\ymax)--(0.04411745,\ymin);
\draw[smooth,black] (-0.8188977,\ymax)--(0.16407181,\ymin);
\draw[smooth,black] (-0.6103866,\ymax)--(0.30190121,\ymin);
\draw[smooth,black] (-0.3115126,\ymax)--(0.46794233,\ymin);
\draw[smooth,black] (0.1138602,\ymax)--(0.66475661,\ymin);
\draw[smooth,black] (0.4367769,\ymax)--(0.81299596,\ymin);
\draw[smooth,black] (0.746287,\ymax)--(0.97317648,\ymin);

    \foreach \Point in {(-1.5875571,\ymax),	(-1.5073526,\ymax),	(-1.4438297,\ymax),	(-1.3869439,\ymax),	(-1.3214492,\ymax),	(-1.2524939,\ymax),	(-1.1844134,\ymax),	(-1.0902795,\ymax),	(-0.9723277,\ymax),	(-0.8188977,\ymax),	(-0.6103866,\ymax),	(-0.3115126,\ymax),	(0.1138602,\ymax),	(0.4367769,\ymax),	(0.746287,\ymax)}
    {\node[blue] at \Point {\scalebox{0.8}{\textbullet}};}

    \foreach \Point in {(-0.88602495,\ymin),	(-0.69198188,\ymin),	(-0.56310801,\ymin),	(-0.4542124,\ymin),	(-0.34841317,\ymin),	(-0.25380006,\ymin),	(-0.16672851,\ymin),	(-0.0645832,\ymin),	(0.04411745,\ymin),	(0.16407181,\ymin),	(0.30190121,\ymin),	(0.46794233,\ymin),	(0.66475661,\ymin),	(0.81299596,\ymin),	(0.97317648,\ymin)}	
    {\node[red] at \Point {\scalebox{0.8}{\textbullet}};}
    
    \foreach \Point in {(-0.76706027,\ymin),	(-0.57517469,\ymin),	(-0.44357328,\ymin),	(-0.33724488,\ymin),	(-0.24438821,\ymin),	(-0.15931968,\ymin),	(-0.07865534,\ymin),	(0,\ymin),	(0.07865534,\ymin),	(0.15931968,\ymin),	(0.24438821,\ymin),	(0.33724488,\ymin),	(0.44357328,\ymin),	(0.57517469,\ymin),	(0.76706027,\ymin)}	
    {\node[green] at \Point {\scalebox{1.5}{$\star$}};}
\end{tikzpicture}
\caption{$N=15,\quad\gamma=0.1$}\label{fig:opt_arrival_departure_1a}
\end{subfigure}

\bigskip

\begin{subfigure}{\columnwidth}
\begin{tikzpicture}[xscale=1.9,yscale=1.6]
  \def\xmin{-3.5}
  \def\xmax{3.5}
  \def\ymin{0}
  \def\ymax{1}
    \draw[->] (\xmin,\ymin) -- (\xmax,\ymin) node[right] 
    {$\left\{\def\arraystretch{0.8}\begin{array}{c} \textcolor{red}{d_i=\bullet} \\ \textcolor{green}{d_i^*=\star} \end{array}\right.$} ;
    \draw[->] (\xmin,\ymax) -- (\xmax,\ymax) node[right] {$\textcolor{blue}{a_i=\bullet}$} ;
     \foreach \x in {-3,-2,-1,0,1,2,3}
    	\node at (\x,\ymin) [below] {\x};
    
\draw[smooth,black] (-1.55712068,\ymax)--(-1.2680003,\ymin);
\draw[smooth,black] (-1.44345342,\ymax)--(-1.0866241,\ymin);
\draw[smooth,black] (-1.35080819,\ymax)--(-0.9466784,\ymin);
\draw[smooth,black] (-1.23855835,\ymax)--(-0.7992043,\ymin);
\draw[smooth,black] (-1.06881044,\ymax)--(-0.5895157,\ymin);
\draw[smooth,black] (-0.94504264,\ymax)--(-0.4465012,\ymin);
\draw[smooth,black] (-0.76603086,\ymax)--(-0.2674814,\ymin);
\draw[smooth,black] (-0.48430565,\ymax)--(-0.0062453,\ymin);
\draw[smooth,black] (-0.26790878,\ymax)--(0.1786758,\ymin);
\draw[smooth,black] (-0.08653336,\ymax)--(0.3241266,\ymin);
\draw[smooth,black] (0.20068554,\ymax)--(0.5629188,\ymin);
\draw[smooth,black] (0.41262452,\ymax)--(0.7346174,\ymin);
\draw[smooth,black] (0.73240807,\ymax)--(0.9887159,\ymin);
\draw[smooth,black] (0.99384417,\ymax)--(1.1912356,\ymin);
\draw[smooth,black] (1.30998006,\ymax)--(1.4300715,\ymin);

 \foreach \Point in {(-1.55712068,\ymax),	(-1.44345342,\ymax),	(-1.35080819,\ymax),	(-1.23855835,\ymax),	(-1.06881044,\ymax),	(-0.94504264,\ymax),	(-0.76603086,\ymax),	(-0.48430565,\ymax),	(-0.26790878,\ymax),	(-0.08653336,\ymax),	(0.20068554,\ymax),	(0.41262452,\ymax),	(0.73240807,\ymax),	(0.99384417,\ymax),	(1.30998006,\ymax)}
    {\node[blue] at \Point {\scalebox{0.8}{\textbullet}};}

    \foreach \Point in {(-1.2680003,\ymin),	(-1.0866241,\ymin),	(-0.9466784,\ymin),	(-0.7992043,\ymin),	(-0.5895157,\ymin),	(-0.4465012,\ymin),	(-0.2674814,\ymin),	(-0.0062453,\ymin),	(0.1786758,\ymin),	(0.3241266,\ymin),	(0.5629188,\ymin),	(0.7346174,\ymin),	(0.9887159,\ymin),	(1.1912356,\ymin),	(1.4300715,\ymin)}	
    {\node[red] at \Point {\scalebox{0.8}{\textbullet}};}
    
    \foreach \Point in {(-0.76706027,\ymin),	(-0.57517469,\ymin),	(-0.44357328,\ymin),	(-0.33724488,\ymin),	(-0.24438821,\ymin),	(-0.15931968,\ymin),	(-0.07865534,\ymin),	(0,\ymin),	(0.07865534,\ymin),	(0.15931968,\ymin),	(0.24438821,\ymin),	(0.33724488,\ymin),	(0.44357328,\ymin),	(0.57517469,\ymin),	(0.76706027,\ymin)}	
    {\node[green] at \Point {\scalebox{1.5}{$\star$}};}
    
\end{tikzpicture}
\caption{$N=15,\quad\gamma=1$}\label{fig:opt_arrival_departure_1b}
\end{subfigure}

\bigskip

\begin{subfigure}{\columnwidth}
\begin{tikzpicture}[xscale=1.9,yscale=1.6]
  \def\xmin{-3.5}
  \def\xmax{3.5}
  \def\ymin{0}
  \def\ymax{1}
    \draw[->] (\xmin,\ymin) -- (\xmax,\ymin) node[right] 
    {$\left\{\def\arraystretch{0.8}\begin{array}{c} \textcolor{red}{d_i=\bullet} \\ \textcolor{green}{d_i^*=\star} \end{array}\right.$} ;
    \draw[->] (\xmin,\ymax) -- (\xmax,\ymax) node[right] {$\textcolor{blue}{a_i=\bullet}$} ;
     \foreach \x in {-3,-2,-1,0,1,2,3}
    	\node at (\x,\ymin) [below] {\x};
    
\draw[smooth,black] (-2.9550767,\ymax)--(-2.935000582,\ymin);
\draw[smooth,black] (-2.3113216,\ymax)--(-2.23341924,\ymin);
\draw[smooth,black] (-1.9351066,\ymax)--(-1.834742404,\ymin);
\draw[smooth,black] (-1.6160803,\ymax)--(-1.496669664,\ymin);
\draw[smooth,black] (-1.2333256,\ymax)--(-1.113902574,\ymin);
\draw[smooth,black] (-0.8348559,\ymax)--(-0.71542603,\ymin);
\draw[smooth,black] (-0.4967773,\ymax)--(-0.377347325,\ymin);
\draw[smooth,black] (-0.1140166,\ymax)--(0.005413805,\ymin);
\draw[smooth,black] (0.284459,\ymax)--(0.403889551,\ymin);
\draw[smooth,black] (0.6225377,\ymax)--(0.741961302,\ymin);
\draw[smooth,black] (1.0052983,\ymax)--(1.124710001,\ymin);
\draw[smooth,black] (1.403774,\ymax)--(1.523173666,\ymin);
\draw[smooth,black] (1.7420466,\ymax)--(1.842393374,\ymin);
\draw[smooth,black] (2.1247953,\ymax)--(2.203583592,\ymin);
\draw[smooth,black] (2.84228,\ymax)--(2.861555608,\ymin);

\foreach \Point in {(-2.9550767,\ymax),	(-2.3113216,\ymax),	(-1.9351066,\ymax),	(-1.6160803,\ymax),	(-1.2333256,\ymax),	(-0.8348559,\ymax),	(-0.4967773,\ymax),	(-0.1140166,\ymax),	(0.284459,\ymax),	(0.6225377,\ymax),	(1.0052983,\ymax),	(1.403774,\ymax),	(1.7420466,\ymax),	(2.1247953,\ymax),	(2.84228,\ymax)}
    {\node[blue] at \Point {\scalebox{0.8}{\textbullet}};}

    \foreach \Point in {(-2.935000582,\ymin),	(-2.23341924,\ymin),	(-1.834742404,\ymin),	(-1.496669664,\ymin),	(-1.113902574,\ymin),	(-0.71542603,\ymin),	(-0.377347325,\ymin),	(0.005413805,\ymin),	(0.403889551,\ymin),	(0.741961302,\ymin),	(1.124710001,\ymin),	(1.523173666,\ymin),	(1.842393374,\ymin),	(2.203583592,\ymin),	(2.861555608,\ymin)}	
    {\node[red] at \Point {\scalebox{0.8}{\textbullet}};}
    
    \foreach \Point in {(-0.76706027,\ymin),	(-0.57517469,\ymin),	(-0.44357328,\ymin),	(-0.33724488,\ymin),	(-0.24438821,\ymin),	(-0.15931968,\ymin),	(-0.07865534,\ymin),	(0,\ymin),	(0.07865534,\ymin),	(0.15931968,\ymin),	(0.24438821,\ymin),	(0.33724488,\ymin),	(0.44357328,\ymin),	(0.57517469,\ymin),	(0.76706027,\ymin)}	
    {\node[green] at \Point {\scalebox{1.5}{$\star$}};}

\end{tikzpicture}
\caption{$N=15,\quad \gamma=20$}\label{fig:opt_arrival_departure_1c}
\end{subfigure}

\begin{subfigure}{\columnwidth}
\begin{tikzpicture}[xscale=1.9,yscale=1.6]
  \def\xmin{-3.5}
  \def\xmax{3.5}
  \def\ymin{0}
  \def\ymax{1}
    \draw[->] (\xmin,\ymin) -- (\xmax,\ymin) node[right] 
    {$\left\{\def\arraystretch{0.8}\begin{array}{c} \textcolor{red}{d_i=\bullet} \\ \textcolor{green}{d_i^*=\star} \end{array}\right.$} ;
    \draw[->] (\xmin,\ymax) -- (\xmax,\ymax) node[right] {$\textcolor{blue}{a_i=\bullet}$} ;
     \foreach \x in {-3,-2,-1,0,1,2,3}
    	\node at (\x,\ymin) [below] {\x};
    	
\draw[smooth,black] (-1.82882604,\ymax)--(-1.14646463,\ymin);
\draw[smooth,black] (-1.77234903,\ymax)--(-0.99971466,\ymin);
\draw[smooth,black] (-1.73562712,\ymax)--(-0.90657558,\ymin);
\draw[smooth,black] (-1.70446545,\ymax)--(-0.82962533,\ymin);
\draw[smooth,black] (-1.68077652,\ymax)--(-0.77237834,\ymin);
\draw[smooth,black] (-1.65651796,\ymax)--(-0.71672635,\ymin);
\draw[smooth,black] (-1.63728509,\ymax)--(-0.67362885,\ymin);
\draw[smooth,black] (-1.61591005,\ymax)--(-0.62805564,\ymin);
\draw[smooth,black] (-1.5982364,\ymax)--(-0.59238323,\ymin);
\draw[smooth,black] (-1.57818441,\ymax)--(-0.55264384,\ymin);
\draw[smooth,black] (-1.5580292,\ymax)--(-0.51480862,\ymin);
\draw[smooth,black] (-1.54158139,\ymax)--(-0.48511135,\ymin);
\draw[smooth,black] (-1.52155767,\ymax)--(-0.45015468,\ymin);
\draw[smooth,black] (-1.50105204,\ymax)--(-0.41620275,\ymin);
\draw[smooth,black] (-1.48418571,\ymax)--(-0.38969522,\ymin);
\draw[smooth,black] (-1.46293258,\ymax)--(-0.35697216,\ymin);
\draw[smooth,black] (-1.44101228,\ymax)--(-0.32493479,\ymin);
\draw[smooth,black] (-1.41815844,\ymax)--(-0.29321229,\ymin);
\draw[smooth,black] (-1.39936684,\ymax)--(-0.26786421,\ymin);
\draw[smooth,black] (-1.37502339,\ymax)--(-0.23646469,\ymin);
\draw[smooth,black] (-1.34936164,\ymax)--(-0.20502001,\ymin);
\draw[smooth,black] (-1.32222114,\ymax)--(-0.17341807,\ymin);
\draw[smooth,black] (-1.29896446,\ymax)--(-0.14761939,\ymin);
\draw[smooth,black] (-1.26918922,\ymax)--(-0.11546214,\ymin);
\draw[smooth,black] (-1.23736693,\ymax)--(-0.08281329,\ymin);
\draw[smooth,black] (-1.20324052,\ymax)--(-0.04955084,\ymin);
\draw[smooth,black] (-1.16651399,\ymax)--(-0.0155448,\ymin);
\draw[smooth,black] (-1.12684445,\ymax)--(0.01934528,\ymin);
\draw[smooth,black] (-1.09163811,\ymax)--(0.04921232,\ymin);
\draw[smooth,black] (-1.04608224,\ymax)--(0.08619296,\ymin);
\draw[smooth,black] (-0.99626602,\ymax)--(0.12455718,\ymin);
\draw[smooth,black] (-0.94155042,\ymax)--(0.16451801,\ymin);
\draw[smooth,black] (-0.88117213,\ymax)--(0.20631837,\ymin);
\draw[smooth,black] (-0.81421307,\ymax)--(0.25023775,\ymin);
\draw[smooth,black] (-0.73956096,\ymax)--(0.29660064,\ymin);
\draw[smooth,black] (-0.65585768,\ymax)--(0.34578711,\ymin);
\draw[smooth,black] (-0.561431,\ymax)--(0.39824637,\ymin);
\draw[smooth,black] (-0.45420313,\ymax)--(0.45451446,\ymin);
\draw[smooth,black] (-0.33156659,\ymax)--(0.51523741,\ymin);
\draw[smooth,black] (-0.19021315,\ymax)--(0.58120235,\ymin);
\draw[smooth,black] (-0.06527597,\ymax)--(0.63759884,\ymin);
\draw[smooth,black] (0.06991414,\ymax)--(0.69716942,\ymin);
\draw[smooth,black] (0.17731494,\ymax)--(0.74428202,\ymin);
\draw[smooth,black] (0.29017729,\ymax)--(0.7944517,\ymin);
\draw[smooth,black] (0.407513,\ymax)--(0.84851728,\ymin);
\draw[smooth,black] (0.50232918,\ymax)--(0.89431988,\ymin);
\draw[smooth,black] (0.61020135,\ymax)--(0.9490498,\ymin);
\draw[smooth,black] (0.71264489,\ymax)--(1.00425512,\ymin);
\draw[smooth,black] (0.8500828,\ymax)--(1.08372289,\ymin);
\draw[smooth,black] (1.02968823,\ymax)--(1.19772885,\ymin);
    
    \foreach \Point in {(-1.82882604,\ymax),	(-1.77234903,\ymax),	(-1.73562712,\ymax),	(-1.70446545,\ymax),	(-1.68077652,\ymax),	(-1.65651796,\ymax),	(-1.63728509,\ymax),	(-1.61591005,\ymax),	(-1.5982364,\ymax),	(-1.57818441,\ymax),	(-1.5580292,\ymax),	(-1.54158139,\ymax),	(-1.52155767,\ymax),	(-1.50105204,\ymax),	(-1.48418571,\ymax),	(-1.46293258,\ymax),	(-1.44101228,\ymax),	(-1.41815844,\ymax),	(-1.39936684,\ymax),	(-1.37502339,\ymax),	(-1.34936164,\ymax),	(-1.32222114,\ymax),	(-1.29896446,\ymax),	(-1.26918922,\ymax),	(-1.23736693,\ymax),	(-1.20324052,\ymax),	(-1.16651399,\ymax),	(-1.12684445,\ymax),	(-1.09163811,\ymax),	(-1.04608224,\ymax),	(-0.99626602,\ymax),	(-0.94155042,\ymax),	(-0.88117213,\ymax),	(-0.81421307,\ymax),	(-0.73956096,\ymax),	(-0.65585768,\ymax),	(-0.561431,\ymax),	(-0.45420313,\ymax),	(-0.33156659,\ymax),	(-0.19021315,\ymax),	(-0.06527597,\ymax),	(0.06991414,\ymax),	(0.17731494,\ymax),	(0.29017729,\ymax),	(0.407513,\ymax),	(0.50232918,\ymax),	(0.61020135,\ymax),	(0.71264489,\ymax),	(0.8500828,\ymax),	(1.02968823,\ymax)}
    {\node[blue] at \Point {\scalebox{0.8}{\textbullet}};}

    \foreach \Point in {(-1.14646463,\ymin),	(-0.99971466,\ymin),	(-0.90657558,\ymin),	(-0.82962533,\ymin),	(-0.77237834,\ymin),	(-0.71672635,\ymin),	(-0.67362885,\ymin),	(-0.62805564,\ymin),	(-0.59238323,\ymin),	(-0.55264384,\ymin),	(-0.51480862,\ymin),	(-0.48511135,\ymin),	(-0.45015468,\ymin),	(-0.41620275,\ymin),	(-0.38969522,\ymin),	(-0.35697216,\ymin),	(-0.32493479,\ymin),	(-0.29321229,\ymin),	(-0.26786421,\ymin),	(-0.23646469,\ymin),	(-0.20502001,\ymin),	(-0.17341807,\ymin),	(-0.14761939,\ymin),	(-0.11546214,\ymin),	(-0.08281329,\ymin),	(-0.04955084,\ymin),	(-0.0155448,\ymin),	(0.01934528,\ymin),	(0.04921232,\ymin),	(0.08619296,\ymin),	(0.12455718,\ymin),	(0.16451801,\ymin),	(0.20631837,\ymin),	(0.25023775,\ymin),	(0.29660064,\ymin),	(0.34578711,\ymin),	(0.39824637,\ymin),	(0.45451446,\ymin),	(0.51523741,\ymin),	(0.58120235,\ymin),	(0.63759884,\ymin),	(0.69716942,\ymin),	(0.74428202,\ymin),	(0.7944517,\ymin),	(0.84851728,\ymin),	(0.89431988,\ymin),	(0.9490498,\ymin),	(1.00425512,\ymin),	(1.08372289,\ymin),	(1.19772885,\ymin)}	
    {\node[red] at \Point {\scalebox{0.8}{\textbullet}};}
    
    \foreach \Point in {(-1.03095825,\ymin),	(-0.87993051,\ymin),	(-0.78236324,\ymin),	(-0.70785105,\ymin),	(-0.64640261,\ymin),	(-0.59341572,\ymin),	(-0.54636791,\ymin),	(-0.5037178,\ymin),	(-0.46444975,\ymin),	(-0.42785622,\ymin),	(-0.39342255,\ymin),	(-0.36076114,\ymin),	(-0.32957152,\ymin),	(-0.29961493,\ymin),	(-0.27069754,\ymin),	(-0.24265887,\ymin),	(-0.21536365,\ymin),	(-0.18869597,\ymin),	(-0.16255486,\ymin),	(-0.13685094,\ymin),	(-0.11150392,\ymin),	(-0.08644042,\ymin),	(-0.06159239,\ymin),	(-0.03689564,\ymin),	(-0.01228863,\ymin),	(0.01228863,\ymin),	(0.03689564,\ymin),	(0.06159239,\ymin),	(0.08644042,\ymin),	(0.11150392,\ymin),	(0.13685094,\ymin),	(0.16255486,\ymin),	(0.18869597,\ymin),	(0.21536365,\ymin),	(0.24265887,\ymin),	(0.27069754,\ymin),	(0.29961493,\ymin),	(0.32957152,\ymin),	(0.36076114,\ymin),	(0.39342255,\ymin),	(0.42785622,\ymin),	(0.46444975,\ymin),	(0.5037178,\ymin),	(0.54636791,\ymin),	(0.59341572,\ymin),	(0.64640261,\ymin),	(0.70785105,\ymin),	(0.78236324,\ymin),	(0.87993051,\ymin),	(1.03095825,\ymin)}	
    {\node[green] at \Point {\scalebox{1.5}{$\star$}};}
\end{tikzpicture}
\caption{$N=50,\quad \gamma=0.1$}\label{fig:opt_arrival_departure_2a}
\end{subfigure}

 \bigskip
 
\begin{subfigure}{\columnwidth}
\begin{tikzpicture}[xscale=1.9,yscale=1.6]
  \def\xmin{-3.5}
  \def\xmax{3.5}
  \def\ymin{0}
  \def\ymax{1}
    \draw[->] (\xmin,\ymin) -- (\xmax,\ymin) node[right] 
    {$\left\{\def\arraystretch{0.8}\begin{array}{c} \textcolor{red}{d_i=\bullet} \\ \textcolor{green}{d_i^*=\star} \end{array}\right.$} ;
    \draw[->] (\xmin,\ymax) -- (\xmax,\ymax) node[right] {$\textcolor{blue}{a_i=\bullet}$} ;
     \foreach \x in {-3,-2,-1,0,1,2,3}
    	\node at (\x,\ymin) [below] {\x};
    
\draw[smooth,black] (-1.80831395,\ymax)--(-1.5407740366,\ymin);
\draw[smooth,black] (-1.73559169,\ymax)--(-1.4221760081,\ymin);
\draw[smooth,black] (-1.67740981,\ymax)--(-1.3285556754,\ymin);
\draw[smooth,black] (-1.63298963,\ymax)--(-1.2575656203,\ymin);
\draw[smooth,black] (-1.59559988,\ymax)--(-1.1991669375,\ymin);
\draw[smooth,black] (-1.56032627,\ymax)--(-1.1455667152,\ymin);
\draw[smooth,black] (-1.52209531,\ymax)--(-1.0884658587,\ymin);
\draw[smooth,black] (-1.47929079,\ymax)--(-1.0256488228,\ymin);
\draw[smooth,black] (-1.45222305,\ymax)--(-0.986631627,\ymin);
\draw[smooth,black] (-1.41620396,\ymax)--(-0.9369318,\ymin);
\draw[smooth,black] (-1.3797996,\ymax)--(-0.8876246187,\ymin);
\draw[smooth,black] (-1.34161644,\ymax)--(-0.8368773061,\ymin);
\draw[smooth,black] (-1.31104585,\ymax)--(-0.7980010512,\ymin);
\draw[smooth,black] (-1.25723417,\ymax)--(-0.7309025443,\ymin);
\draw[smooth,black] (-1.21888731,\ymax)--(-0.6840365497,\ymin);
\draw[smooth,black] (-1.17981666,\ymax)--(-0.6383731913,\ymin);
\draw[smooth,black] (-1.11935429,\ymax)--(-0.5691692723,\ymin);
\draw[smooth,black] (-1.07362554,\ymax)--(-0.5179333691,\ymin);
\draw[smooth,black] (-1.02880368,\ymax)--(-0.4699451873,\ymin);
\draw[smooth,black] (-0.95873791,\ymax)--(-0.3965822071,\ymin);
\draw[smooth,black] (-0.89348502,\ymax)--(-0.3297962301,\ymin);
\draw[smooth,black] (-0.83701451,\ymax)--(-0.2738661859,\ymin);
\draw[smooth,black] (-0.76593366,\ymax)--(-0.2060510462,\ymin);
\draw[smooth,black] (-0.67629979,\ymax)--(-0.1225988223,\ymin);
\draw[smooth,black] (-0.58643658,\ymax)--(-0.0410011116,\ymin);
\draw[smooth,black] (-0.54067663,\ymax)--(-0.0008756756,\ymin);
\draw[smooth,black] (-0.4431564,\ymax)--(0.0834980898,\ymin);
\draw[smooth,black] (-0.35475281,\ymax)--(0.1595106199,\ymin);
\draw[smooth,black] (-0.30326162,\ymax)--(0.2034683076,\ymin);
\draw[smooth,black] (-0.22282614,\ymax)--(0.270527066,\ymin);
\draw[smooth,black] (-0.14566946,\ymax)--(0.3352935063,\ymin);
\draw[smooth,black] (-0.08856868,\ymax)--(0.3825705627,\ymin);
\draw[smooth,black] (-0.02555164,\ymax)--(0.4347458369,\ymin);
\draw[smooth,black] (0.0629653,\ymax)--(0.5075352937,\ymin);
\draw[smooth,black] (0.11247248,\ymax)--(0.5487023127,\ymin);
\draw[smooth,black] (0.20189597,\ymax)--(0.6223311962,\ymin);
\draw[smooth,black] (0.26919447,\ymax)--(0.6773934905,\ymin);
\draw[smooth,black] (0.36152374,\ymax)--(0.7523419165,\ymin);
\draw[smooth,black] (0.43092766,\ymax)--(0.8082692271,\ymin);
\draw[smooth,black] (0.52995165,\ymax)--(0.8874574317,\ymin);
\draw[smooth,black] (0.60351463,\ymax)--(0.9458951951,\ymin);
\draw[smooth,black] (0.69317342,\ymax)--(1.0162340611,\ymin);
\draw[smooth,black] (0.79384568,\ymax)--(1.0945454611,\ymin);
\draw[smooth,black] (0.8774979,\ymax)--(1.1589501547,\ymin);
\draw[smooth,black] (0.97539756,\ymax)--(1.2332985741,\ymin);
\draw[smooth,black] (1.0833947,\ymax)--(1.3150446519,\ymin);
\draw[smooth,black] (1.15960723,\ymax)--(1.3720438262,\ymin);
\draw[smooth,black] (1.27042357,\ymax)--(1.4546600683,\ymin);
\draw[smooth,black] (1.38246486,\ymax)--(1.5383177315,\ymin);
\draw[smooth,black] (1.50762976,\ymax)--(1.6326063826,\ymin);

 \foreach \Point in {(-1.80831395,\ymax),	(-1.73559169,\ymax),	(-1.67740981,\ymax),	(-1.63298963,\ymax),	(-1.59559988,\ymax),	(-1.56032627,\ymax),	(-1.52209531,\ymax),	(-1.47929079,\ymax),	(-1.45222305,\ymax),	(-1.41620396,\ymax),	(-1.3797996,\ymax),	(-1.34161644,\ymax),	(-1.31104585,\ymax),	(-1.25723417,\ymax),	(-1.21888731,\ymax),	(-1.17981666,\ymax),	(-1.11935429,\ymax),	(-1.07362554,\ymax),	(-1.02880368,\ymax),	(-0.95873791,\ymax),	(-0.89348502,\ymax),	(-0.83701451,\ymax),	(-0.76593366,\ymax),	(-0.67629979,\ymax),	(-0.58643658,\ymax),	(-0.54067663,\ymax),	(-0.4431564,\ymax),	(-0.35475281,\ymax),	(-0.30326162,\ymax),	(-0.22282614,\ymax),	(-0.14566946,\ymax),	(-0.08856868,\ymax),	(-0.02555164,\ymax),	(0.0629653,\ymax),	(0.11247248,\ymax),	(0.20189597,\ymax),	(0.26919447,\ymax),	(0.36152374,\ymax),	(0.43092766,\ymax),	(0.52995165,\ymax),	(0.60351463,\ymax),	(0.69317342,\ymax),	(0.79384568,\ymax),	(0.8774979,\ymax),	(0.97539756,\ymax),	(1.0833947,\ymax),	(1.15960723,\ymax),	(1.27042357,\ymax),	(1.38246486,\ymax),	(1.50762976,\ymax)}
    {\node[blue] at \Point {\scalebox{0.8}{\textbullet}};}

    \foreach \Point in {(-1.5407740366,\ymin),	(-1.4221760081,\ymin),	(-1.3285556754,\ymin),	(-1.2575656203,\ymin),	(-1.1991669375,\ymin),	(-1.1455667152,\ymin),	(-1.0884658587,\ymin),	(-1.0256488228,\ymin),	(-0.986631627,\ymin),	(-0.9369318,\ymin),	(-0.8876246187,\ymin),	(-0.8368773061,\ymin),	(-0.7980010512,\ymin),	(-0.7309025443,\ymin),	(-0.6840365497,\ymin),	(-0.6383731913,\ymin),	(-0.5691692723,\ymin),	(-0.5179333691,\ymin),	(-0.4699451873,\ymin),	(-0.3965822071,\ymin),	(-0.3297962301,\ymin),	(-0.2738661859,\ymin),	(-0.2060510462,\ymin),	(-0.1225988223,\ymin),	(-0.0410011116,\ymin),	(-0.0008756756,\ymin),	(0.0834980898,\ymin),	(0.1595106199,\ymin),	(0.2034683076,\ymin),	(0.270527066,\ymin),	(0.3352935063,\ymin),	(0.3825705627,\ymin),	(0.4347458369,\ymin),	(0.5075352937,\ymin),	(0.5487023127,\ymin),	(0.6223311962,\ymin),	(0.6773934905,\ymin),	(0.7523419165,\ymin),	(0.8082692271,\ymin),	(0.8874574317,\ymin),	(0.9458951951,\ymin),	(1.0162340611,\ymin),	(1.0945454611,\ymin),	(1.1589501547,\ymin),	(1.2332985741,\ymin),	(1.3150446519,\ymin),	(1.3720438262,\ymin),	(1.4546600683,\ymin),	(1.5383177315,\ymin),	(1.6326063826,\ymin)}	
    {\node[red] at \Point {\scalebox{0.8}{\textbullet}};}
    
    \foreach \Point in {(-1.03095825,\ymin),	(-0.87993051,\ymin),	(-0.78236324,\ymin),	(-0.70785105,\ymin),	(-0.64640261,\ymin),	(-0.59341572,\ymin),	(-0.54636791,\ymin),	(-0.5037178,\ymin),	(-0.46444975,\ymin),	(-0.42785622,\ymin),	(-0.39342255,\ymin),	(-0.36076114,\ymin),	(-0.32957152,\ymin),	(-0.29961493,\ymin),	(-0.27069754,\ymin),	(-0.24265887,\ymin),	(-0.21536365,\ymin),	(-0.18869597,\ymin),	(-0.16255486,\ymin),	(-0.13685094,\ymin),	(-0.11150392,\ymin),	(-0.08644042,\ymin),	(-0.06159239,\ymin),	(-0.03689564,\ymin),	(-0.01228863,\ymin),	(0.01228863,\ymin),	(0.03689564,\ymin),	(0.06159239,\ymin),	(0.08644042,\ymin),	(0.11150392,\ymin),	(0.13685094,\ymin),	(0.16255486,\ymin),	(0.18869597,\ymin),	(0.21536365,\ymin),	(0.24265887,\ymin),	(0.27069754,\ymin),	(0.29961493,\ymin),	(0.32957152,\ymin),	(0.36076114,\ymin),	(0.39342255,\ymin),	(0.42785622,\ymin),	(0.46444975,\ymin),	(0.5037178,\ymin),	(0.54636791,\ymin),	(0.59341572,\ymin),	(0.64640261,\ymin),	(0.70785105,\ymin),	(0.78236324,\ymin),	(0.87993051,\ymin),	(1.03095825,\ymin)}	
    {\node[green] at \Point {\scalebox{1.5}{$\star$}};}
    
\end{tikzpicture}
\caption{$N=50,\quad \gamma=1$}\label{fig:opt_arrival_departure_2b}
\end{subfigure}

\bigskip
 
\begin{subfigure}{\columnwidth}
\begin{tikzpicture}[xscale=1.9,yscale=1.6]
  \def\xmin{-3.5}
  \def\xmax{3.5}
  \def\ymin{0}
  \def\ymax{1}
    \draw[->] (\xmin,\ymin) -- (\xmax,\ymin) node[right] 
    {$\left\{\def\arraystretch{0.8}\begin{array}{c} \textcolor{red}{d_i=\bullet} \\ \textcolor{green}{d_i^*=\star} \end{array}\right.$} ;
    \draw[->] (\xmin,\ymax) -- (\xmax,\ymax) node[right] {$\textcolor{blue}{a_i=\bullet}$} ;
     \foreach \x in {-3,-2,-1,0,1,2,3}
    	\node at (\x,\ymin) [below] {\x};
    
\draw[smooth,black] (-3.17606026,\ymax)--(-3.1253581,\ymin);
\draw[smooth,black] (-3.06048684,\ymax)--(-2.99751134,\ymin);
\draw[smooth,black] (-2.87667122,\ymax)--(-2.79534456,\ymin);
\draw[smooth,black] (-2.74927388,\ymax)--(-2.65646991,\ymin);
\draw[smooth,black] (-2.55768783,\ymax)--(-2.44851659,\ymin);
\draw[smooth,black] (-2.42290212,\ymax)--(-2.30383828,\ymin);
\draw[smooth,black] (-2.2540116,\ymax)--(-2.1238118,\ymin);
\draw[smooth,black] (-2.12525987,\ymax)--(-1.98784034,\ymin);
\draw[smooth,black] (-1.99761317,\ymax)--(-1.85303573,\ymin);
\draw[smooth,black] (-1.91310688,\ymax)--(-1.76379077,\ymin);
\draw[smooth,black] (-1.79524643,\ymax)--(-1.64152053,\ymin);
\draw[smooth,black] (-1.65657184,\ymax)--(-1.49671346,\ymin);
\draw[smooth,black] (-1.59055562,\ymax)--(-1.42694777,\ymin);
\draw[smooth,black] (-1.44841856,\ymax)--(-1.28091189,\ymin);
\draw[smooth,black] (-1.30394033,\ymax)--(-1.13093155,\ymin);
\draw[smooth,black] (-1.2240502,\ymax)--(-1.04800175,\ymin);
\draw[smooth,black] (-1.12371388,\ymax)--(-0.94575426,\ymin);
\draw[smooth,black] (-0.9877443,\ymax)--(-0.80719477,\ymin);
\draw[smooth,black] (-0.85313984,\ymax)--(-0.67002824,\ymin);
\draw[smooth,black] (-0.76389301,\ymax)--(-0.57908528,\ymin);
\draw[smooth,black] (-0.64162469,\ymax)--(-0.4544918,\ymin);
\draw[smooth,black] (-0.547452,\ymax)--(-0.35852914,\ymin);
\draw[smooth,black] (-0.4956141,\ymax)--(-0.30667217,\ymin);
\draw[smooth,black] (-0.34726389,\ymax)--(-0.15680602,\ymin);
\draw[smooth,black] (-0.2808161,\ymax)--(-0.09035826,\ymin);
\draw[smooth,black] (-0.13083584,\ymax)--(0.05883618,\ymin);
\draw[smooth,black] (-0.04790419,\ymax)--(0.14022149,\ymin);
\draw[smooth,black] (0.05434139,\ymax)--(0.24055968,\ymin);
\draw[smooth,black] (0.19290088,\ymax)--(0.37653301,\ymin);
\draw[smooth,black] (0.33006733,\ymax)--(0.51113738,\ymin);
\draw[smooth,black] (0.42101028,\ymax)--(0.60038047,\ymin);
\draw[smooth,black] (0.54560371,\ymax)--(0.72264505,\ymin);
\draw[smooth,black] (0.64156827,\ymax)--(0.81509168,\ymin);
\draw[smooth,black] (0.69342524,\ymax)--(0.86504934,\ymin);
\draw[smooth,black] (0.84329135,\ymax)--(1.00941948,\ymin);
\draw[smooth,black] (0.95099269,\ymax)--(1.11392721,\ymin);
\draw[smooth,black] (1.14011701,\ymax)--(1.2976008,\ymin);
\draw[smooth,black] (1.2404552,\ymax)--(1.39432319,\ymin);
\draw[smooth,black] (1.37642853,\ymax)--(1.52539469,\ymin);
\draw[smooth,black] (1.51123298,\ymax)--(1.65533945,\ymin);
\draw[smooth,black] (1.60047607,\ymax)--(1.7398475,\ymin);
\draw[smooth,black] (1.81498728,\ymax)--(1.94460311,\ymin);
\draw[smooth,black] (1.86494491,\ymax)--(1.99278879,\ymin);
\draw[smooth,black] (2.00931502,\ymax)--(2.12962472,\ymin);
\draw[smooth,black] (2.11402463,\ymax)--(2.22886943,\ymin);
\draw[smooth,black] (2.39421881,\ymax)--(2.49155663,\ymin);
\draw[smooth,black] (2.52549039,\ymax)--(2.61400308,\ymin);
\draw[smooth,black] (2.73974332,\ymax)--(2.81194426,\ymin);
\draw[smooth,black] (2.94469897,\ymax)--(3.00023672,\ymin);
\draw[smooth,black] (3.1297207,\ymax)--(3.1696888,\ymin);

\foreach \Point in {(-3.17606026,\ymax),	(-3.06048684,\ymax),	(-2.87667122,\ymax),	(-2.74927388,\ymax),	(-2.55768783,\ymax),	(-2.42290212,\ymax),	(-2.2540116,\ymax),	(-2.12525987,\ymax),	(-1.99761317,\ymax),	(-1.91310688,\ymax),	(-1.79524643,\ymax),	(-1.65657184,\ymax),	(-1.59055562,\ymax),	(-1.44841856,\ymax),	(-1.30394033,\ymax),	(-1.2240502,\ymax),	(-1.12371388,\ymax),	(-0.9877443,\ymax),	(-0.85313984,\ymax),	(-0.76389301,\ymax),	(-0.64162469,\ymax),	(-0.547452,\ymax),	(-0.4956141,\ymax),	(-0.34726389,\ymax),	(-0.2808161,\ymax),	(-0.13083584,\ymax),	(-0.04790419,\ymax),	(0.05434139,\ymax),	(0.19290088,\ymax),	(0.33006733,\ymax),	(0.42101028,\ymax),	(0.54560371,\ymax),	(0.64156827,\ymax),	(0.69342524,\ymax),	(0.84329135,\ymax),	(0.95099269,\ymax),	(1.14011701,\ymax),	(1.2404552,\ymax),	(1.37642853,\ymax),	(1.51123298,\ymax),	(1.60047607,\ymax),	(1.81498728,\ymax),	(1.86494491,\ymax),	(2.00931502,\ymax),	(2.11402463,\ymax),	(2.39421881,\ymax),	(2.52549039,\ymax),	(2.73974332,\ymax),	(2.94469897,\ymax),	(3.1297207,\ymax)}
    {\node[blue] at \Point {\scalebox{0.8}{\textbullet}};}

    \foreach \Point in {(-3.1253581,\ymin),	(-2.99751134,\ymin),	(-2.79534456,\ymin),	(-2.65646991,\ymin),	(-2.44851659,\ymin),	(-2.30383828,\ymin),	(-2.1238118,\ymin),	(-1.98784034,\ymin),	(-1.85303573,\ymin),	(-1.76379077,\ymin),	(-1.64152053,\ymin),	(-1.49671346,\ymin),	(-1.42694777,\ymin),	(-1.28091189,\ymin),	(-1.13093155,\ymin),	(-1.04800175,\ymin),	(-0.94575426,\ymin),	(-0.80719477,\ymin),	(-0.67002824,\ymin),	(-0.57908528,\ymin),	(-0.4544918,\ymin),	(-0.35852914,\ymin),	(-0.30667217,\ymin),	(-0.15680602,\ymin),	(-0.09035826,\ymin),	(0.05883618,\ymin),	(0.14022149,\ymin),	(0.24055968,\ymin),	(0.37653301,\ymin),	(0.51113738,\ymin),	(0.60038047,\ymin),	(0.72264505,\ymin),	(0.81509168,\ymin),	(0.86504934,\ymin),	(1.00941948,\ymin),	(1.11392721,\ymin),	(1.2976008,\ymin),	(1.39432319,\ymin),	(1.52539469,\ymin),	(1.65533945,\ymin),	(1.7398475,\ymin),	(1.94460311,\ymin),	(1.99278879,\ymin),	(2.12962472,\ymin),	(2.22886943,\ymin),	(2.49155663,\ymin),	(2.61400308,\ymin),	(2.81194426,\ymin),	(3.00023672,\ymin),	(3.1696888,\ymin)}	
    {\node[red] at \Point {\scalebox{0.8}{\textbullet}};}
    
    \foreach \Point in {(-1.03095825,\ymin),	(-0.87993051,\ymin),	(-0.78236324,\ymin),	(-0.70785105,\ymin),	(-0.64640261,\ymin),	(-0.59341572,\ymin),	(-0.54636791,\ymin),	(-0.5037178,\ymin),	(-0.46444975,\ymin),	(-0.42785622,\ymin),	(-0.39342255,\ymin),	(-0.36076114,\ymin),	(-0.32957152,\ymin),	(-0.29961493,\ymin),	(-0.27069754,\ymin),	(-0.24265887,\ymin),	(-0.21536365,\ymin),	(-0.18869597,\ymin),	(-0.16255486,\ymin),	(-0.13685094,\ymin),	(-0.11150392,\ymin),	(-0.08644042,\ymin),	(-0.06159239,\ymin),	(-0.03689564,\ymin),	(-0.01228863,\ymin),	(0.01228863,\ymin),	(0.03689564,\ymin),	(0.06159239,\ymin),	(0.08644042,\ymin),	(0.11150392,\ymin),	(0.13685094,\ymin),	(0.16255486,\ymin),	(0.18869597,\ymin),	(0.21536365,\ymin),	(0.24265887,\ymin),	(0.27069754,\ymin),	(0.29961493,\ymin),	(0.32957152,\ymin),	(0.36076114,\ymin),	(0.39342255,\ymin),	(0.42785622,\ymin),	(0.46444975,\ymin),	(0.5037178,\ymin),	(0.54636791,\ymin),	(0.59341572,\ymin),	(0.64640261,\ymin),	(0.70785105,\ymin),	(0.78236324,\ymin),	(0.87993051,\ymin),	(1.03095825,\ymin)}	
    {\node[green] at \Point {\scalebox{1.5}{$\star$}};}

\end{tikzpicture}
\subcaption{$N=50,\quad \gamma=20$}\label{fig:opt_arrival_departure_2c}
\end{subfigure}

\caption{Optimal arrival-departure diagram for $\alpha=0.8/N$, $\beta=1$, and ${\mathbf d}^*$ the $N$ quantiles of a normal distribution with mean $0$ and standard deviation $1/2$.}\label{fig:opt_arrival_departure2}
\end{figure}

\begin{table}[H]
\centering
\tiny
\caption{Running time in seconds and computational steps of the combined heuristic (Algorithm ~6 with $M=1$) and the exhaustive search (Algorithm ~2).\label{tblLife}}
\begin{tabular}{|c|c|c|c|c|c|c|c|c|c|c|} \hline 
 $N$  & 3 & 5  & 10 & 11 & 12 & 14 & 15 & 20 & 30 & 50 \\ \hline 
 \multicolumn{11}{|l|}{\textbf{Combined heuristic}}  \\ \hline
CPI cycles & 3 & 2 & 4 & 5 & 4 & 3 & 3 & 3 & 3 & 4 \\ \hline
Total breakpoints & 24 & 65  & 306 & 382 & 441 & 642 & 727 & 1,383 & 3,260 & 8,636 \\ \hline
NS QPs solved & 2 & 2 & 9 & 6 & 8 & 14 & 10 & 29 & 34 & 29 \\ \hline
Running time (sec.) & 0.05 & 0.15 & 1.17 & 1.67 & 1.99 & 3.09 & 3.28 & 6.38 & 18.31 & 85.33 \\ \hline
Global opt. & Yes & Yes & Yes & Yes & Yes & Yes & Yes & NA & NA & NA \\ \hline  
 \multicolumn{11}{|l|}{\textbf{Exhaustive search}}  \\ \hline
$|\mathcal{K}|$ QPs solved & 5 & 42 & 16,796 & 58,786 & 208,012 & $ 2.6\cdot 10^6$ & $9.7\cdot 10^6$ & $ 6.6\cdot 10^9$ & $ 3.8\cdot 10^{15}$ & $ 2\cdot 10^{27}$ \\ \hline
Running time (sec.) & 0.00 & 0.05 & 25.25 & 162.41 & 509 & 8,206 & 39,454 & NA & NA & NA \\ \hline
\end{tabular}
\label{tbl:iter_compare}
\end{table}


To further investigate our combined heuristic, in Figure \ref{fig:alg_iterations} we illustrate the number CPI cycles and breakpoints, along with the respective number of quadratic programs solved by the neighbour search, until convergence of Algorithm 6. The problem data was scaled as in the previous example. For every $N$ the initial points given by $\mathcal{A}$ with $M=5$ distinct initial points. The figure displays the minimum and maximum values out of the $5$ initial points. Note that for every $N$ the algorithm converged to the same local minimum for all initial points in $\mathcal{A}$. 

We can see that the number of required CPI cycles was small and stabilized on $2$ regardless of the number of users. However, we should take into account that the number of coordinate iterations in every cycle is $N$, and that the complexity of each iteration also grows with $N$. Specifically, Proposition \ref{prop:alg4_complexity} shows that the number of breakpoints for every coordinate in the CPI is at most $N^3$, but in the example we see the growth is in effect linear ($\sim 3N$). Furthermore, the number of required quadratic programs solved in the neighbour search also grows linearly ($\sim \frac{1}{3}N$). This hints that the CPI does indeed find a point that is very ``close" to a local minimum. The widening gap between the minimum and maximum number of NS iterations suggests that some of the initial points are better than others, and thus it is worthwhile trying several of them. The last point is important when solving for even larger values of $N$ as the algorithm becomes more sensitive to ``bad" initial points and may require setting a maximum number of iterations parameter for every initial point. {Roughly, when $\gamma$ and $\alpha$ are both small, starting closer to $\mathbf{a}^0$ is better and when they are both large, starting closer to $\mathbf{a}^\infty$ is better. However, for most combinations of parameters there seems to be no a-priori indication of what is a ``good" starting point. Thus it is still beneficial to do the full search on $\mathcal{A}$.} Again we stress that the behaviour displayed in Figure \ref{fig:alg_iterations} was robust with respect to changes in the model parameters.

\begin{figure}[H]
\begin{subfigure}{.45\linewidth}
\begin{tikzpicture}[y=0.45cm,x=0.02cm]
  \def\xmin{0}
  \def\xmax{305}
  \def\ymin{0}
  \def\ymax{8}
   
    \draw[->] (\xmin,\ymin) -- (\xmax,\ymin) node[right] {$N$} ;
    \draw[->] (0,\ymin) -- (0,\ymax) node[above] {} ;
    \foreach \x in {0,50,100,150,200,250,300}
    \node at (\x,\ymin) [below] {\x};
    \foreach \y in {2,4,6,8}
    \node at (0,\y) [left] {\y};

	\draw[smooth,red,thick] (10,4)--	(20,6)--	(50,4)--	(75,6)--	(100,2)--	(125,2)--	(150,2)--	(175,2)--	(200,2)--	(225,2)--	(250,2)--	(275,2)--	(300,2);

	    \foreach \Point in {(10,4),	(20,6),	(50,4),	(75,6),	(100,2),	(125,2),	(150,2),	(175,2),	(200,2),	(225,2),	(250,2),	(275,2),	(300,2)}
    {\node[red] at \Point {\textbullet};}
	
	\draw[smooth,blue,thick] (10,3)--	(20,3)--	(50,3)--	(75,2)--	(100,2)--	(125,2)--	(150,2)--	(175,2)--	(200,2)--	(225,2)--	(250,2)--	(275,2)--	(300,2);

	    \foreach \Point in {(10,3),	(20,3),	(50,3),	(75,2),	(100,2),	(125,2),	(150,2),	(175,2),	(200,2),	(225,2),	(250,2),	(275,2),	(300,2)}
    {\node[blue] at \Point {$\star$};}
	
\end{tikzpicture}
\subcaption{CPI Cycles}\label{fig:alg_iterations_a}
\end{subfigure}
\begin{subfigure}{.45\linewidth}
\begin{tikzpicture}[y=0.0035cm,x=0.02cm]
  \def\xmin{0}
  \def\xmax{305}
  \def\ymin{0}
  \def\ymax{1000}
  
    \draw[->] (\xmin,\ymin) -- (\xmax,\ymin) node[right] {$N$} ;
    \draw[->] (0,\ymin) -- (0,\ymax) node[above] {} ;
    \foreach \x in {0,50,100,150,200,250,300}
    	\node at (\x,\ymin) [below] {\x};
    \foreach \y in {100,200,300,400,500,600,700,800,900}
    	\node at (0,\y) [left] {\y};

	\draw[smooth,red,thick] (10,101)--	(20,83)--	(50,189)--	(75,264)--	(100,337)--	(125,412)--	(150,481)--	(175,564)--	(200,642)--	(225,724)--	(250,792)--	(275,869)--	(300,946);

		    \foreach \Point in {(10,101),	(20,83),	(50,189),	(75,264),	(100,337),	(125,412),	(150,481),	(175,564),	(200,642),	(225,724),	(250,792),	(275,869),	(300,946)}
    {\node[red] at \Point {\textbullet};}
    
    \draw[smooth,blue,thick] (10,25)--	(20,56)--	(50,146)--	(75,219)--	(100,295)--	(125,369)--	(150,444)--	(175,517)--	(200,587)--	(225,667)--	(250,743)--	(275,809)--	(300,885);

		    \foreach \Point in {(10,25),	(20,56),	(50,146),	(75,219),	(100,295),	(125,369),	(150,444),	(175,517),	(200,587),	(225,667),	(250,743),	(275,809),	(300,885)}
    {\node[blue] at \Point {$\star$};}	
    
\end{tikzpicture}
\subcaption{CPI Breakpoints}\label{fig:alg_iterations_b}
\end{subfigure}

\bigskip 
\centering
\begin{subfigure}{.5\linewidth}
\begin{tikzpicture}[y=0.03cm,x=0.02cm]
  \def\xmin{0}
  \def\xmax{305}
  \def\ymin{0}
  \def\ymax{120}
    \draw[->] (\xmin,\ymin) -- (\xmax,\ymin) node[right] {$N$} ;
    \draw[->] (0,\ymin) -- (0,\ymax) node[above] {} ;
    \foreach \x in {0,50,100,150,200,250,300}
    \node at (\x,\ymin) [below] {\x};
    \foreach \y in {20,40,60,80,100,120}
    \node at (0,\y) [left] {\y};

	\draw[smooth,red,thick] (10,10)--	(20,50)--	(50,51)--	(75,97)--	(100,59)--	(125,59)--	(150,58)--	(175,86)--	(200,75)--	(225,108)--	(250,89)--	(275,111)--	(300,115);	
	
		    \foreach \Point in {(10,10),	(20,50),	(50,51),	(75,97),	(100,59),	(125,59),	(150,58),	(175,86),	(200,75),	(225,108),	(250,89),	(275,111),	(300,115)}
    {\node[red] at \Point {\textbullet};}
    
    \draw[smooth,blue,thick] (10,7)--	(20,28)--	(50,21)--	(75,48)--	(100,17)--	(125,14)--	(150,15)--	(175,24)--	(200,28)--	(225,24)--	(250,22)--	(275,21)--	(300,20);

		    \foreach \Point in {(10,7),	(20,28),	(50,21),	(75,48),	(100,17),	(125,14),	(150,15),	(175,24),	(200,28),	(225,24),	(250,22),	(275,21),	(300,20)}
    {\node[blue] at \Point {$\star$};}
	
\end{tikzpicture}
\subcaption{NS iterations}\label{fig:alg_iterations_c}
\end{subfigure}
\begin{subfigure}[b]{.2\linewidth}
\begin{tikzpicture}[scale=10]
    \begin{customlegend}
    [legend entries={Maximum, Minimum},
    legend style={at={(1.4,0.5)}}]
	    \addlegendimage{red,thick,mark=*}
    	\addlegendimage{blue,thick,mark=star}
    \end{customlegend}
\end{tikzpicture}
\end{subfigure}
\caption{Number of iterations in each component of Algorithm 6 as a function of $N$.}
\label{fig:alg_iterations}
\end{figure}
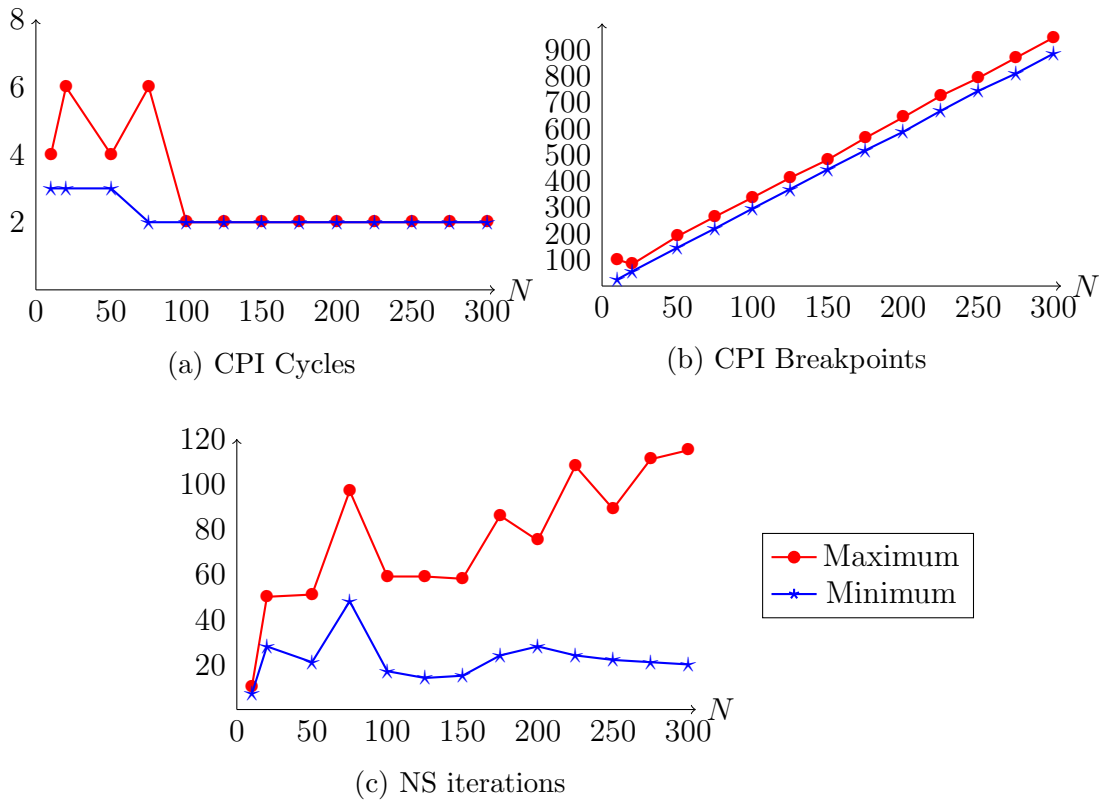

\section{Conclusion and Outlook}
\label{sec:Conclusion}
We presented a model for a discrete-user deterministic processor sharing system, and addressed the problem of scheduling arrivals to such a system with the goal of minimizing congestion and tardiness costs. A full characterisation of the congestion dynamics and an efficient method for computing them was provided. It was further shown that the optimal arrival schedule can be computed in a finite, but exponentially large, number of steps. Several heuristics were therefore developed with the goal of an efficient computation of the optimal schedule. A combined global and local search heuristic was presented and numerically analysed. This method was shown to be efficient in numerical examples for a large population of users.

The essential parts of our analysis and results applies for a much more general cost formulation, as we shall next detail. Given that user $i$ enters the system at time $a_i$ and leaves at time $d_i > a_i$, a plausible cost incurred by the user is the following:
\begin{align}\label{eq:costUser}
c_i(a_i,d_i) &= g_i^{(1)}\Big((d_i-d_i^*)^+\Big)\,+\, g_i^{(2)}\Big((d_i^*-d_i)^+\Big) \\ 
 & +g_i^{(3)}\Big((a_i-a_i^*)^+\Big)\,+\, g_i^{(4)}\Big((a_i^*-a_i)^+\Big) \nonumber \\ 
 & +g_i^{(5)}\Big(d_i - a_i \Big), \nonumber
\end{align}
where $(x)^+ := \max(x,0)$, and $g_i^{(j)}(\cdot), j\in\{1,\dots 5\},~i \in {\cal N}$ are some convex functions.

The first and third terms of \eqref{eq:costUser} capture the penalty for being late to the ideal departure and arrival times $d_i^*$ and $a_i^*$, respectively. The second and fourth terms are the user's cost for arriving and departing early. The fifth term is the user's cost for travel/usage of the system. Our algorithm and results in this paper hold with slight technical modifications for arbitrary convex $g_i^{(j)}(\cdot)$. For purpose of exposition, we focused on, $g_i^{(1)}(x)=g_i^{(2)}(x)=x^2$, $g_i^{(3)}(x)=g_i^{(4)}(x)=0$ and $g_i^{(5)}(x)=\gamma \, x$. If adapted to the more general formulation, The exhaustive and neighbour search algorithms of Section \ref{sec:optRegion} will generally require solving a constrained convex program, instead of convex quadratic, for every region. If $g_i^{(1)}(x)\neq g_i^{(2)}(x)$ and/or $g_i^{(3)}(x)\neq g_i^{(4)}(x)$, namely there are different penalties for arriving/departing later and early, then the CPI algorithm of Section \ref{sec:CPIAlg} will require some refinement of the definition of the piecewise segments. The complexity will not change as for every single coordinate there will be an addition of at most three segments, corresponding for these new points of discontinuity. Moreover, the root of the first order condition in every continuous segment will be given by the general form of the functions, instead of the quadratic root. 

An interesting generalization is considering a system with users who have heterogeneous service demand. If this is the case then the order of departures is no longer identical to the order of arrivals. This means that the characterisation of Proposition \ref{Lemma_departure_function} is no longer valid. 

A natural complementary model to this work is considering a decentralized decision framework in which the users choose their own arrival time. Namely, a non cooperative game with the individual arrival times are the actions of the players. This game is formulated and analysed in \cite{RHV2016}.

Finally, there is the challenge of characterising the computational complexity of our scheduling problem. We believe that finding the optimal ${\mathbf k}^* \in {\cal K}$ is an NP-complete problem but we still do not have a proof for this. Our belief is motivated (but not supported) by the fact that there are a number of related optimization problems which are known to be NP hard. Our problem is equivalent to a special case of one of them, namely non-linear integer programming.

As we have shown, our goal is to minimize a non-convex piecewise quadratic objective function, subject to piecewise linear constraints. It is known that non-convex quadratic programs and non-convex piecewise linear optimization are both NP hard (see \cite{KFM2006} and \cite{MK1987}). In \cite{VAN2010} it is shown that piecewise linear optimization problems can be modelled as linear mixed integer programs, where the definition of a piecewise linear program relies on different coefficients for different polytopes, in a similar manner to our piecewise quadratic formulation in Section \ref{sec:optRegion}. It may be possible to apply similar methods with modifications for the piecewise convex instead of linear objective. However, there is a more natural construction for our case. Let $\tilde{\mathbf{a}}(\mathbf{k})$ be the solution to $QP(\mathbf{k})$, i.e., the solution to the local convex QP of a polytope $\mathbf{k}\in\mathcal{K}$. But this can also be viewed as a function of the integer vector $\mathbf{k}$ which we can compute in polynomial time. Hence, solving our problem in polynomial time is equivalent to solving the non-linear integer program:
\[
\min_{\mathbf{k}\in\mathcal{K}}\tilde{\mathbf{a}}(\mathbf{k})'Q_{\mathbf k} \tilde{\mathbf{a}}(\mathbf{k})+\mathbf{b}_{\mathbf k} \, \tilde{\mathbf{a}}(\mathbf{k})+{\tilde{b}}_{\mathbf k}.
\]

Recall that $\mathcal{K}=\left\lbrace\mathbf{k}\in{\cal N}^N \,:\, k_N=N,\, k_i \le k_{j} \ \forall i \le j\right\rbrace$ defines a set of linear constraints on the integer decision variables. Clearly the objective is not linear with respect to $\mathbf{k}$, as $\tilde{\mathbf{a}}(\mathbf{k})$ itself is already not necessarily linear. Such problems are known to be NP hard. See for example, \cite{LHK2006} and \cite{PDM2016}. Although we could not find a straightforward reduction of the problem to a known NP hard problem, we have shown that our problem can be formulated as an (rather cumbersome) instance of a polynomial integer program, and have no reason to believe that the specific model comes with significant simplification of the general form. 

As a closing note we mention that it is generally of interest to compare our heuristics to potential integer programming methods. One may either discretize time and solve integer programs, or alternatively seek related integer programming formulations. It remains an open problem to compare our heuristics to such potential methods both in terms of accuracy and computation time.


{\bf Acknowledgements:} We thank Hai Vu and Moshe Haviv for useful discussions and advice. We are grateful to two anonymous reviewers for their helpful comments. We thank The Australia-Israel Scientific Exchange Foundation (AISEF) for supporting Liron Ravner's visit to The University of Queensland. Yoni Nazarathy's research is supported by ARC grants DE130100291 and DP130100156.

\appendix
\section{Proofs}

{\bf Proof of Lemma~\ref{Lemma:departure_order}}:
Consider two arrivals $a_{i} \le a_{j}$. During the time interval $[a_{i},a_{j}]$, user $i$ has received some service,
\[
\int_{a_{i}}^{a_{j}} v\big(q(t)\big) dt,
\]
while user $j$ has not. Then during the time interval $[a_{j}, d_{i} \wedge d_{j}]$ both users receive the same service, $\int_{a_{j}}^{d_{i} \wedge d_{j}} v\big(q(t)\big) dt$. Then if $d_i > d_j$ we have that $\int_{a_{j}}^{d_{i} \wedge d_{j}} v\big(q(t)\big) dt = 1$, which in turn would imply that,
\[
\int_{a_{i}}^{d_{i}} v\big(q(t)\big) dt =  \int_{a_{i}}^{a_{j}} v\big(q(t)\big) dt + \int_{a_{j}}^{d_{i} \wedge d_{j}} v\big(q(t)\big) dt +\int_{d_{i} \wedge d_{j}}^{d_{i}} v\big(q(t)\big) dt > 1, 
\]
a contradiction. Hence $d_i \le d_j$.
\qed
\vspace{5mm}

\noindent
{\bf Proof of Lemma~\ref{lemma:unique}}:
Without loss of generality assume $a_1 \le \ldots \le a_N$ and hence by the previous lemma, ${\mathbf d}$ is ordered. Assume now that there exists a $\tilde{\mathbf d} \neq {\mathbf d}$ and define $i = \min\{ i \, : \, \tilde{d}_i \neq d_i \}$. Without loss of generality, assume that $d_i < \tilde{d}_i$. Using \eqref{eq:dynamics} it holds that,
\[
\int_{a_i}^{d_i} v \big(q(t)\big) \, dt =1= \int_{a_i}^{d_i} v \big( \tilde{q}(t) \big) \, dt + \int_{d_i}^{\tilde{d}_i} v \big( \tilde{q}(t) \big) \, dt.
\]
Now since for all $t \le d_i$ it holds that $q(t) = \tilde{q}(t)$, then,
\[
\int_{d_i}^{\tilde{d}_i} v \big( \tilde{q}(t) \big) \, dt = 0.
\]
A contradiction. \\

Now there exists a full symmetry between ${\mathbf a}$ and ${\mathbf d}$, hence going in the opposite direction (for every ${\mathbf d}$ there exists a unique ${\mathbf a}$) follows a similar argument to the above.
\qed
\vspace{5mm}

\noindent
{\bf Proof of Lemma~\ref{Lemma:optimal_order}}:
We first argue that an optimal arrival must be ordered ($ a_1 \le \ldots \le a_N$) by means of an interchange argument. Assume this is not the case, i.e. $\mathbf{a}$ is an optimal arrival schedule such that $a_i>a_j$ for some $i<j$ (such that $d_i^*<d_j^*$). If we switch between the arrival times of users $i$ and $j$: $\tilde{a_i}=a_j$ and $\tilde{a_j}=a_i$, while not changing any other arrival time, then because all users have the same service demand the departure times of all other users do not change. Consequently, the departure times are also switched: $\tilde{d_i}=d_j$ and $\tilde{d_j}=d_i$. Therefore, the only change in the total cost function is the change in the cost incurred by $i$ and $j$ themselves. The change in the cost incurred by user $i$ is given by \eqref{eq:costUserShort}:
\begin{equation}\label{eq:cost_change_i}
\begin{split}
{c_i}(\tilde{a_i},\tilde{d_i})-c_i(a_i,d_i) &=\big(\tilde{d_i}-d_i^*\big)^2+\gamma\big(\tilde{d_i}-\tilde{a_i}\big)-(d_i-d_i^*)^2-\gamma(d_i-a_i) \\
&= (d_j-d_i^*)^2+\gamma(d_j-a_j)-(d_i-d_i^*)^2-\gamma(d_i-a_i)
\end{split},
\end{equation}
and for user $j$:
\begin{equation}\label{eq:cost_change_j}
\begin{split}
{c_j}(\tilde{a}_j,\tilde{d}_j)-c_j(a_j,d_j) &=\big(\tilde{d_j}-d_j^*\big)^2+\gamma\big(\tilde{d_j}-\tilde{a}_j\big)-(d_j-d_j^*)^2-\gamma(d_j-a_j) \\
&= (d_i-d_j^*)^2+\gamma(d_i-a_i)-(d_j-d_j^*)^2-\gamma(d_j-a_j)
\end{split}.
\end{equation}
Summing \eqref{eq:cost_change_i} and \eqref{eq:cost_change_j} we obtain that the total change in cost is 
\begin{equation*}
2\big(d_i^*-d_j^*\big)\big(d_i-d_j\big).
\end{equation*} 
From Lemma \ref{Lemma:departure_order} we know that if $a_i>a_j$ then $d_i>d_j$, and that by definition $d_j^*> d_i^*$, hence the change in the total cost function is negative which contradicts the assumption that the schedule is optimal. In conclusion, any unordered schedule can be improved by a simple interchange of a pair of unordered coordinates, and therefore an optimal schedule must be ordered.

The slowest service rate occurs when all $N$ users are present in the system, and therefore the longest possible sojourn time is $\frac{1}{\beta-\alpha(N-1)}$. The total time required to clear all users from the system is then coarsely upper bounded by $\frac{N}{\beta-\alpha(N-1)}$. A schedule such that $a_1<\underline{a}$ is clearly not optimal, since a trivial improvement can always be achieved by setting $a_1=\underline{a}$ and shifting to the right the arrival times of any user that overlap due to the change in $a_1$. We are guaranteed this is possible by the fact that all users can arrive and leave the system in the interval $[\underline{a},d_1^*]$, without any overlaps. Clearly, the deviation from ideal times can only decrease when making this change, while the sojourn times remain unchanged. The coarse upper bound $\overline{a}$ holds for the same reasons. ~\qed
\vspace{5mm}

\noindent
{\bf Proof of Proposition~\ref{prop:alg1_complexity}}:
The proof is for Algorithm~1a. The argument for Algorithm~1b follows the same arguments.
For every user $i$, iterating on all possible values of $k_i\in\mathcal{N}$ ensures that every possible departure interval $[a_k,a_{k+1})$ is checked. In a sense, this is an exhaustive search on all solutions that satisfy the dynamics given by Proposition \ref{Lemma_departure_function}. Therefore, the algorithm will always converge to the unique solution.

Given a vector of arrivals ${\mathbf a} \in \mathbbm{R}^N$, for every $i\in\mathcal{N}$, the departure time $d_i$ occurs in one of the above defined partitions $[a_k,a_{k+1}), \ k\in\mathcal{N}$. The total number of steps will include the number of ``correct'' computations, that is for every $i$ and $k_i=k$ the resulting $d_i$ will indeed be in the interval $[a_k,a_{k+1})$. In total there will be exactly $N$ correct computations. However, there will also be steps which will turn out to be false: for a given $k$ the departure time $d_i$ will not be in the interval $[a_k,a_{k+1})$. If $k_j=k$ for some $j$, then for every $i>j$: $k_i\geq k$. Therefore, if for some $i$ and $k\geq i$ the computation will yield $d_i\notin[a_k,a_{k+1})$, then this interval will not be attempted by any later arrival $j>i$ in the following steps. As a result, every interval will yield at most one false computation. Since there are exactly $N$ intervals this completes the proof. 
\qed

\vspace{5mm}

\noindent
{\bf Proof of Lemma~\ref{Lemma_coeff_decreasing}}:
Since $x>a_i$ it holds that $i < \pi_r$, and thus using \eqref{eq:theta_single} if $k_i<\pi_r$ then $\theta_i=0$, and if $k_i\geq \pi_r$ then
\[
\theta_i=-\frac{\alpha\left(1-\sum_{j=h_i}^{i-1}\theta_j\right)}{\beta-\alpha(k_i-i)}.
\]
Since $N< \beta/\alpha+1$, the denominator is always positive. We next show that the numerator is non-negative by induction on $h_{\pi_r}\leq i <\pi_r$. Recall that $h_i=\min\lbrace h:k_h\geq i\rbrace$, and so $k_i\geq \pi_r$ is equivalent to $i\geq h_{\pi_r}$. Thus for $j<h_{\pi_r}$: $\theta_j=0$ and the denominator in the case $i=h_{\pi_r}$ equals $\alpha (1-0)>0$. The induction step is then immediate because the sum $\sum_{j=h_i}^{i-1}\theta_j$ is non-negative for all $h_{\pi_r}<i <\pi_r$. 
\qed
\vspace{5mm}

\noindent
{\bf Proof of Lemma~\ref{Lemma_coeff_changes}}:
Without loss of generality assume that $\underline{a}+\frac{1}{\beta}<a_i<\overline{a}-\frac{1}{\beta}, \ \forall i\in\mathcal{N}$. If this were not the case we could always extend the search range by $\frac{1}{\beta}$ in both directions. Hence, $\theta_i=0$ at $x=\underline{a}$ and at $x=\overline{a}$ for any $i\in\pi$. Furthermore, from Lemma \ref{Lemma_coeff_decreasing} we have that $\theta_i\leq 0$ for $x>a_i$. Clearly, there is some $x$ such that $\theta_i>0$ for the first time. So far we have established that $\theta_i$ starts at zero, is positive at some point and negative at some back to zero, for every $i\in\pi$. We are left with finding the number of possible sign changes prior to $x=a_i$. For any $x<a_i$ it follows that $i>\pi_r$, and from \eqref{eq:theta_single} we have that:
\[
\theta_i=\frac{\alpha\sum_{j=h_i}^{i-1}\theta_j}{\beta-\alpha(k_i-i)}
\]
Note that $\theta_i$ can only be negative when there is at least one $j<i$ such that $\theta_j<0$. We use this to complete the proof by induction on the initial order $\pi$. We start the induction at $i=2$ because $\pi_r=1$ in the initial permutation and $\theta_{\pi_r}\geq 0$ for all values of $x$. For $i=2$ and $x<a_2$: $\theta_2=\frac{\alpha\theta_{1}\mathbbm{1}\{h_2=1\}}{\beta-\alpha(k_2-2)}\geq 0$. Together with Lemma \ref{Lemma_coeff_decreasing} we have established that $\theta_2$ changes sign exactly once. Now let us assume that the claim is correct for all $j\leq i-1$. From \eqref{eq:theta_single} we see that for $x<a_i$, $\theta_j$ can only change sign when one of the previous $j\in\{h_i,\dots ,i-1\}$ changes sign. If $\theta_{i-1}$ changed sign exactly $i-2$ times then $\theta_i$ can potentially change at all these times and additionally when $x=a_i$, and therefore there are indeed at most $i-1$ changes of sign.
\qed

\footnotesize{\bibliography{PaperDatabase,BookDatabase}}

\begin{thebibliography}{10}

\bibitem{ADL1993}
R.~Arnott, A.~de~Palma, and R.~Lindsey.
\newblock A structural model of peak-period congestion: A traffic bottleneck
  with elastic demand.
\newblock {\em American Economic Review}, 83(1):161--79, 1993.

\bibitem{avram1995fluid}
F.~Avram, D.~Bertsimas, and M.~Ricard.
\newblock {Fluid models of sequencing problems in open queueing networks; an
  optimal control approach}.
\newblock {\em Institute for Mathematics and Its Applications}, 71:199, 1995.

\bibitem{baker1990sequencing}
K.~Baker and G.~D. Scudder.
\newblock Sequencing with earliness and tardiness penalties: a review.
\newblock {\em Operations Research}, 38(1):22--36, 1990.

\bibitem{Bertsekas1999}
D.~P. Bertsekas.
\newblock {\em {Nonlinear Programming}}.
\newblock Athena Scientific, 2nd edition, Sept. 1999.

\bibitem{cohen1979multiple}
J.~Cohen.
\newblock The multiple phase service network with generalized processor
  sharing.
\newblock {\em Acta Informatica}, 12(3):245--284, 1979.

\bibitem{Daganzo2007}
C.~F. Daganzo.
\newblock Urban gridlock: Macroscopic modeling and mitigation approaches.
\newblock {\em Transportation Research Part B: Methodological}, 41(1):49--62,
  2007.

\bibitem{LHK2006}
J.~A. De~Loera, R.~Hemmecke, M.~Köppe, and R.~Weismantel.
\newblock Integer polynomial optimization in fixed dimension.
\newblock {\em Mathematics of Operations Research}, 31(1):147--153, 2006.

\bibitem{GH1983}
A.~Glazer and R.~Hassin.
\newblock {?/M/1: On the equilibrium distribution of customer arrivals}.
\newblock {\em European Journal of Operational Research}, 13(2):146 -- 150,
  1983.

\bibitem{harchol2013performance}
M.~Harchol-Balter.
\newblock {\em Performance Modeling and Design of Computer Systems: Queueing
  Theory in Action}.
\newblock Cambridge University Press, 2013.

\bibitem{book_H2016}
R.~Hassin.
\newblock {\em Rational queueing}.
\newblock CRC Press, 2016.

\bibitem{H1974}
J.~Henderson.
\newblock Road congestion.
\newblock {\em Journal of Urban Economics}, 1(3):346 -- 365, 1974.

\bibitem{KFM2006}
A.~B. Keha, I.~R. de~Farias, and G.~L. Nemhauser.
\newblock A branch-and-cut algorithm without binary variables for nonconvex
  piecewise linear optimization.
\newblock {\em Operations Research}, 54(5):847--858, 2006.

\bibitem{koshy2009catalan}
T.~Koshy.
\newblock {\em Catalan numbers with applications}, volume~10.
\newblock 2009.

\bibitem{MAHMASSANI1984}
H.~Mahmassani and R.~Herman.
\newblock Dynamic user equilibrium departure time and route choice on idealized
  traffic arterials.
\newblock {\em Transportation Science}, 18(4):pp. 362--384, 1984.

\bibitem{MK1987}
K.~G. Murty and S.~N. Kabadi.
\newblock {Some NP-complete problems in quadratic and nonlinear programming}.
\newblock {\em Mathematical Programming}, 39(2):117--129, 1987.

\bibitem{nazarathy2009near}
Y.~Nazarathy and G.~Weiss.
\newblock Near optimal control of queueing networks over a finite time horizon.
\newblock {\em Annals of Operations Research}, 170(1):233--249, 2009.

\bibitem{PDM2016}
A.~D. Pia, S.~S. Dey, and M.~Molinaro.
\newblock "mixed-integer quadratic programming is in np".
\newblock {\em Mathematical Programming}, pages 1--16, 2016.

\bibitem{book_P2008}
M.~L. Pinedo.
\newblock {\em Scheduling: Theory, Algorithms, and Systems}.
\newblock Springer, 2008.

\bibitem{PM2000}
C.~N. Potts and M.~Y. Kovalyov.
\newblock Scheduling with batching: A review.
\newblock {\em European Journal of Operational Research}, 120(2):228 -- 249,
  2000.

\bibitem{RHV2016}
L.~Ravner, M.~Haviv, and H.~L. Vu.
\newblock A strategic timing of arrivals to a linear slowdown processor sharing
  system.
\newblock {\em European Journal of Operational Research}, 255(2):496 -- 504,
  2016.

\bibitem{sen1984state}
T.~Sen and S.~K. Gupta.
\newblock A state-of-art survey of static scheduling research involving due
  dates.
\newblock {\em Omega}, 12(1):63--76, 1984.

\bibitem{Tseng2001}
P.~Tseng.
\newblock Convergence of a block coordinate descent method for
  nondifferentiable minimization.
\newblock {\em Journal of Optimization Theory and Applications},
  109(3):475--494, 2001.

\bibitem{VAN2010}
J.~P. Vielma, S.~Ahmed, and G.~L. Nemhauser.
\newblock Mixed-integer models for nonseparable piecewise-linear optimization:
  Unifying framework and extensions.
\newblock {\em Operations Research}, 58(2):303--315, 2010.

\bibitem{weiss2008simplex}
G.~Weiss.
\newblock A simplex based algorithm to solve separated continuous linear
  programs.
\newblock {\em Mathematical Programming}, 115:151--198, 2008.

\end{thebibliography}
\end{document}